\documentclass[aop,MSNbibl,seceqn,citesort,dvips]{arximspdf}

\doi{10.1214/11-AOP643}
\volume{40}
\issue{3}
\pubyear{2012}
\firstpage{1135}
\lastpage{1166}

\makeatletter

\newtheorem{Lemma}{Lemma}[section]
\newtheorem{Proposition}[Lemma]{Proposition}
\newtheorem{Theorem}[Lemma]{Theorem}
\newtheorem{Corollary}[Lemma]{Corollary}
\newtheorem{Claim}[Lemma]{Claim}

\newproclaim{Notation}[Lemma]{Notation}
\newproclaim{definition}[Lemma]{Definition}
\newproclaim{remark}[Lemma]{Remark}

\newcommand{\gcal}{\mathcal{G}}
\newcommand{\eps}{\varepsilon}

\newcommand{\bbP}{{\mathbb{P}}}
\newcommand{\Real}{{\mathbb{R}}}
\newcommand{\E}{\mathbb{E}}
\newcommand{\prob}{\mathbb{P}}
\newcommand{\jcal}{\mathcal J}

\newcommand{\ga}{\alpha}
\newcommand{\gd}{\delta}
\newcommand{\gL}{\Lambda}
\newcommand{\Ent}{\operatorname{Ent}}
\newcommand{\Var}{\operatorname{Var}}
\newcommand{\ig}{I^{\gcal}}

\makeatother

\begin{document}
\begin{frontmatter}

\title{Geometric influences}
\runtitle{Geometric influences}

\begin{aug}
\author[A]{\fnms{Nathan} \snm{Keller}\thanksref{t1}\ead[label=e1]{nathan.keller@weizmann.ac.il}},
\author[B]{\fnms{Elchanan} \snm{Mossel}\corref{}\thanksref{t2}\ead[label=e2]{mossel@stat.berkeley.edu}} and
\author[C]{\fnms{Arnab} \snm{Sen}\thanksref{t3}\ead[label=e3]{a.sen@statslab.cam.ac.uk}}
\runauthor{N. Keller, E. Mossel and A. Sen}
\affiliation{Weizmann Institute of Science,
Weizmann Institute of Science and University of California, Berkeley,
and Cambridge University}
\address[A]{N. Keller\\
Faculty of Mathematics\\
\quad and Computer Science\\
The Weizmann Institute of Science\\
Rehovot\\
Israel\\
\printead{e1}} 
\address[B]{E. Mossel\\
Department of Statistics\\
University of California, Berkeley\\
367 Evans Hall Berkeley, California 94720\\
USA\\
\printead{e2}}
\address[C]{A. Sen\\
Statistical Laboratory\\
Department of Pure Mathematics\\
\quad and Mathematical Sciences\\
Wilberforce Road, CB3 0WB\\
United Kingdom\\
\printead{e3}}
\end{aug}

\thankstext{t1}{Supported by the Koshland Center for Basic
Research. Some of this work was conducted when the author was at the Hebrew
University of Jerusalem, and was supported by the Adams Fellowship Program
of the Israeli Academy of Sciences of Humanities.}

\thankstext{t2}{Supported by CAREER Award DMS-05-48249, by ISF Grant 1300/08,
by a Minerva Foundation grant and by an ERC Marie Curie Grant 2008 239317.}

\thankstext{t3}{Most of this work was conducted when the author was at
University of California, Berkeley, and was supported by DOD ONR Grant N0014-07-1-05-06.}

\received{\smonth{4} \syear{2010}}
\revised{\smonth{11} \syear{2010}}

%
\begin{abstract}
We present a new definition of influences in product spaces of
continuous distributions. Our definition is geometric, and for monotone
sets it is identical with the measure of the boundary with respect to
uniform enlargement. We prove analogs of the Kahn--Kalai--Linial (KKL)
and Talagrand's influence sum bounds for the new definition. We further
prove an analog of a result of Friedgut showing that sets with small
``influence sum'' are essentially determined by a small number of
coordinates. In particular, we establish the following tight analog of
the KKL bound: for any set in $\mathbb R^n$ of Gaussian measure~$t$,
there exists a coordinate $i$ such that the $i$th geometric influence
of the set is at least $ct(1-t)\sqrt{\log n}/n$, where $c$ is a
universal constant. This result is then used to obtain an isoperimetric
inequality for the Gaussian measure on $\mathbb{R}^n$ and the class of
sets invariant under transitive permutation group of the coordinates.
\end{abstract}

%

\vspace*{-3pt}
\end{frontmatter}

\section{Introduction}\label{sec-int}\vspace*{-3pt}

\begin{definition}
Let $f\dvtx\{0,1\}^n \rightarrow\{0,1\}$ be a Boolean function. The
influence of the $i$th coordinate on $f$ is
\[
I_i(f):= \bbP_{x \sim\mu} [f(x) \neq f(x \oplus e_i)],
\]
where $\bbP_{x \sim\mu}$ denotes probability when $x$ is chosen at
random according to a probability measure $\mu$, and $x\oplus e_i$
denotes the point obtained from $x$ by replacing $x_i$ by $1-x_i$
and leaving the other coordinates unchanged.\vadjust{\goodbreak}
\end{definition}

The notion of influences of variables on Boolean functions is one
of the central concepts in the theory of discrete harmonic
analysis. In the last two decades it found several applications in
diverse fields, including combinatorics, theoretical computer
science, statistical physics, social choice theory, etc.
(see, e.g., the survey article \cite{kalai06}). The influences have numerous
properties that allow us to use them in applications. The following
three properties are among the most fundamental ones:
\begin{longlist}[(2)]
\item[(1)] \textit{Geometric meaning.} The influences on the discrete
cube $\{0,1\}^n$ have a clear geometric meaning. $I_i(f)$ is the
measure of the \textit{edge boundary in the $i$th direction} of the
set $A=\{x \in\{0,1\}^n\dvtx f(x)=1\}$.

\item[(2)] \textit{The KKL theorem.} In the remarkable
paper \cite{kahn88}, Kahn, Kalai and Linial proved that for any
Boolean function $f\dvtx\{0,1\}^n \rightarrow\{0,1\}$, there exists a
variable $i$ whose influence is at least $ct(1-t)\log n/n$, where
$t=\mathbb{E} [f]$ is the expectation of~$f$, and $c$ is a
universal constant. Many applications of influences make use of
the KKL theorem or of related results such
as \cite{Talagrand1,Friedgut1} in one way or another.

\item[(3)] \textit{The Russo lemma.} Let $\mu_p$ denote the Bernoulli
measure where $0$ is given weight $1-p$ and $1$ is given weight
$p$. Clearly if $A \subseteq\{0, 1\}^n$ is monotone increasing
[i.e., satisfies the condition that if $(x_1,\ldots,x_n) \in A$
and $y_j \geq x_j$ for all $j$, then $y=(y_1,\ldots,y_n) \in A$],
then $\mu_p^{\otimes n}(A)$ is monotone increasing as function of
$p$. The question of understanding how $\mu_p^{\otimes n}(A)$
varies with $p$ has important applications in the theory of random
graphs and in percolation theory. Russo's lemma
\cite{margulis74,russo82} asserts that the derivative of $\mu
_p^{\otimes n}(A)$
with respect to $p$ is the sum of influences of $f=1_A$.
\end{longlist}
The basic results on influences were obtained for
functions on the discrete cube, but some applications required
generalization of the results to more general product spaces.
Unlike the discrete case, where there exists a single natural
definition of influence, for general product spaces several
definitions were presented in different
papers (see, e.g., \cite{bourgain92,hatami09,keller09b,mossel09}).
While each of these definitions has its advantages, in
general all of them lack geometric interpretation for continuous
probability spaces.

In this paper we present a new definition of the
influences in product spaces of continuous random variables, that
has a clear geometric meaning. We show that for the Gaussian
measure and for a more general class of log-concave product
measures called Boltzmann measures (see
Definition~\ref{defboltzmann}), our definition allows us to obtain
analogs of the KKL theorem and Russo-type formulas.
\begin{definition}
Let $\nu_i$ be a probability measure on $\mathbb{R}$. Given a
Borel-measurable set $A \subseteq\mathbb{R}$, its lower Minkowski
content $\nu_i^{+}(A)$ is defined as
\[
\nu_i^{+}(A) := \liminf_{ r \downarrow0} \frac{\nu_i( A + [-r,
r]) - \nu_i(A)}{r}.
\]
Consider the product measure $\nu= \nu_1 \otimes\nu_2 \otimes
\cdots\otimes\nu_n$ on $\mathbb{R}^n$. Then for
any~Bo\-rel-measurable set $A \subseteq\mathbb{R}^n$, for each $ 1 \leq
i \leq n$ and an element $x =(x_1, x_2, \ldots,\allowbreak x_n) \in
\mathbb{R}^n$, the restriction of $A$ along the fiber of $x$ in
the $i$th direction is given by
\[
A^{x}_i := \{ y \in\mathbb{R}\dvtx(x_1, \ldots, x_{i-1}, y,
x_{i+1}, \ldots, x_n) \in A \}.
\]
The \textit{geometric influence} of the $i$th coordinate on $A$ is
\[
\ig_i(A) := \mathbb{E}_{x} [\nu_i^{+} (A^{x}_i)],
\]
that is, the expectation of $\nu_i^{+} (A^{x}_i)$ when $x$ is
chosen according to the measure~$\nu$. For sake of clarity,
we sometimes denote the influence as $\ig_i(A) |_{\nu}$.
\end{definition}

The geometric meaning of the influence is that for a monotone
(either increasing or decreasing) set $A$, the sum of influences
of $A$ is equal to the size of its boundary with respect to a
uniform enlargement as shown in the following
proposition.
\begin{Proposition}\label{luniformenlarge}
Let $\nu$ be a probability measure on $\mathbb{R}$ with $C^1$
density~$\lambda$ and cumulative distribution function $\Lambda$.
Assume further that $\lambda(z) > 0$ for all $z \in\mathbb{R}$,
that $\lim_{ |z| \rightarrow\infty} \lambda(z) = 0$ and that~$\lambda'$
is bounded. Let $A \subset\Real^n$ be a~monotone set. Then
\[
\lim_{r \downarrow0} \frac{\nu^{\otimes n} (A+ [-r,r]^n) -
\nu^{\otimes n} (A)}{r} = \sum_{i=1}^n \ig_i(A).
\]
\end{Proposition}

Though the boundary under uniform enlargement is perhaps not as
prevalent in the literature as the usual $L^2$ boundary (where we
fatten a set by adding a $L^2$-ball of radius $r$ instead of a
$L^\infty$-ball), the isoperimetric problem for the boundary under
uniform enlargement, especially in the context of log-concave measures,
was studied, for example, in \cite{bobkov96,bobkov97,barthe04}. Note
that for the Gaussian measure on $\mathbb R^n$, unlike the usual $L^2$
boundary, the boundary under uniform enlargement of a set is not
invariant under rotation.

We show that for the Gaussian measure on $\mathbb{R}^n$,
the geometric influences satisfy the following analog of the KKL
theorem:
\begin{Theorem}\label{ThmOur-KKL}
Consider the product spaces $\mathbb R^n$ endowed with the product
Gaussian measure $\mu^{\otimes n}$.
Then for any Borel-measurable set $A \subset\Real^n$ with
$\mu^{\otimes n}(A)=t$
there exists $1 \leq i \leq n$ such that
\[
\ig_i(A) \ge ct(1-t) \frac{ \sqrt{\log n}}{n},
\]
where $c>0$ is a universal constant.
\end{Theorem}

The result extends to the family of Boltzmann measures (see
Definition~\ref{defboltzmann}), and is tight up to the constant
factor. The proof uses the\vadjust{\goodbreak} relation between geometric influences
and the $h$-influences defined in \cite{keller09b}, combined with
isoperimetric estimates for the underlying probability
measures.

Using the same methods, we obtain analogs of Talagrand's
bound on the vector of influences \cite{Talagrand1} and of Friedgut's
theorem stating that a function with a low sum of influences essentially
depends on a few coordinates \cite{Friedgut1}.
\begin{Theorem}\label{ThmTalagrand-intro}
Consider the product spaces $\mathbb R^n$ endowed with the product
Gaussian measure $\mu^{\otimes n}$. For any Borel-measurable set $A
\subset\Real^n$, we have:
\begin{longlist}[(2)]
\item[(1)] if $\mu^{\otimes n}(A)=t$, then
\[
\sum_{i=1}^n \frac{\ig_i(A)}{\sqrt{-\log\ig_i(A)}} \ge
c_1 t(1-t);
\]

\item[(2)] if $A$ is monotone and $\sum_{i=1}^n \ig_i(A) \sqrt{-\log\ig
_i(A)}=s$,
then there exists a set $B \subset\mathbb{R}^n$ such that $1_B$ is
determined by at
most $\exp(c_2 s/\epsilon)$ coordinates and
$\mu^{\otimes n}(A \bigtriangleup B) \leq\epsilon$,
\end{longlist}
where $c_1$ and $c_2$ are universal constants.
\end{Theorem}

We also show that the geometric influences can be used
in Russo-type formulas for location families.
\begin{Proposition}\label{proprusso-location}
Let $\nu$ be a probability measure on $\mathbb{R}$ with continuous
density $\lambda$ and cumulative distribution function $\Lambda$.
Let $\{ \nu_{\alpha}\dvtx\alpha\in\mathbb{R} \}$ denote a~family
of probability measures which is obtained by translating $\nu$;
that is, $\nu_{\alpha}$ has a density $\lambda_\ga$ satisfying
$\lambda_\ga(x) = \lambda(x - \alpha)$.

Assume that $\lambda$ is bounded and satisfies $\lambda(z)
> 0$ on $(\kappa_L, \kappa_R)$, the interior of the support of
$\nu$. Let $A$ be a monotone increasing subset of $\mathbb{R}^n$.
Then the function $\alpha\rightarrow{\nu_{\alpha}}^{\otimes n}(
A)$ is differentiable, and its derivative is given by
\[
\frac{d {\nu_{\alpha}}^{\otimes n}( A)}{d \alpha} = \sum_{i=1}^n
\ig_i(A),
\]
where the influences are taken with respect to the measure
$\nu_\ga^{\otimes n}$.
\end{Proposition}

Theorem \ref{ThmOur-KKL} and
Proposition \ref{proprusso-location} can be combined to get the
following corollary which is the Gaussian analog of the sharp threshold
result obtained by Friedgut and Kalai
\cite{Friedgut-Kalai} for the product Bernoulli measure on the hypercube.
We call a set transitive if its characteristic function
is invariant under the action of some transitive subgroup of the
permutation group $S_n$, where~$S_n$ acts by permutation of the $n$
coordinates.
\begin{Corollary}\label{CorSharp-Threshold}
Let $\mu_{\ga}$ denote the Gaussian measure on the real line with
mean $\ga$ and variance $1$. Let $A \subset\mathbb{R}^n$ be a
monotone increasing transitive set. For any
$\delta>0$, denote\vadjust{\goodbreak} by $\alpha_A(\delta)$ the unique value of
$\alpha$ such that \mbox{$\mu_{\alpha}^{\otimes n}(A)=\delta$}. Then for
any $0<\epsilon<1/2$,
\[
\alpha_A(1-\epsilon)-\alpha_A(\epsilon) \leq c
\log(1/2\epsilon)/\sqrt{\log n},
\]
where $c$ is a universal constant.
\end{Corollary}

We use the geometric influences to obtain an isoperimetric result
for the Gaussian measure on $\mathbb{R}^n$.
\begin{Theorem}\label{thmisolowerbound}
Consider the product spaces $\mathbb R^n$ endowed with the product
Gaussian measure $\mu^{\otimes n}$. Then for any transitive
Borel-measurable set $A \subset\Real^n$
we have
\[
\liminf_{r \downarrow0} \frac{\mu^{\otimes n} (A+ [-r,r]^n) -
\mu^{\otimes n} (A)}{r} \geq ct(1-t) \sqrt{\log n},
\]
where $t=\mu^{\otimes n}(A)$ and $c>0$ is a universal constant.
\end{Theorem}

This result also extends to all Boltzmann measures.

Since the Gaussian measure is rotation invariant, it is natural to
consider the influence sum of rotations of sets.
Of particular interest are families of sets that are closed under
rotations. In
Section \ref{secrotation} we study the effect of \textit{rotations} on
the geometric influences, and show that under mild
regularity condition of being in a certain class $\jcal_n$ (see
Definition \ref{defjcal}), the sum of geometric influences of a
convex set can be
increased up to $\Omega(\sqrt{n})$ by (a~random) orthogonal rotation.
\begin{Theorem}\label{ThmRotation}
Consider the product Gaussian measure $\mu^{\otimes n}$ on
$\mathbb R^n$. For any convex set $A \in\jcal_n$ with
$\mu^{\otimes n}(A) = t$, we have
\[
\E_{M \sim\pi} \Biggl[ \sum_{i=1}^n \ig_i(M(A)) \Biggr] \geq
ct(1-t) \sqrt{-\log\bigl(t(1-t)\bigr)} \times\sqrt n,
\]
where $\E_{M \sim\pi}$ denotes the expectation when $M$ is drawn
according to the Haar measure $\pi$ over the
orthogonal group of rotations, and $c$ is a
universal positive constant. In particular, there exists an orthogonal
transformation $g$ on $\mathbb R^n$ such that
\[
\sum_{i=1}^n \ig_i(g(A)) \ge ct(1-t) \sqrt{-\log\bigl(t(1-t)\bigr)} \times
\sqrt n.
\]
\end{Theorem}

The results presented in this paper lead us to the questions regarding
the extension in several directions:
\begin{itemize}
\item\textit{Nonproduct measures.} The most challenging
direction of extending our results is to consider nonproduct
probability measures. The first problem in such generalization is
that it is not clear at all what is the natural definition of
influences for such measures. The second difficulty is that the
techniques used in KKL-type results rely quite heavily on\vadjust{\goodbreak}
properties of product measures, and it is not clear how they can
be extended to more general measures. We note that in a recent
paper \cite{Grimmett06}, Graham and Grimmett obtained a variant of
the KKL theorem for measures satisfying certain FKG lattice
conditions
using reduction to the uniform measure on the continuous cube
$[0,1]^n$. Their results hold both for FKG measures on $\{0,1\}^n$
and for FKG measures on $[0,1]^n$ which are absolutely continuous
with respect to the Lebesgue measure. However, it is not clear
whether the definition of influences they consider in the
continuous case has an ``interesting'' geometric interpretation.

\item\textit{Geometric meaning of nonmonotone sets.} The geometric
meaning of our
definition of the influences (i.e., the relation to the size of
the boundary with respect to uniform enlargement) holds for all
monotone sets (as shown in Proposition~\ref{luniformenlarge}),
and even for all convex sets, but not for general sets (see
Remark \ref{remconvex}). While restriction to monotone sets is
quite standard in the study of influences (since monotonization
arguments similar to Lemma \ref{theomono} show that it is
sufficient to prove all the lower bounds on influences in the case
of monotone sets), it is interesting to find out whether our
definition of influences has other geometric interpretation for general
sets. On the other hand, it is interesting to determine exactly the
families of sets to which Proposition \ref{luniformenlarge}
applies.

\item\textit{Other continuous measures.} The main results of our
paper apply to the Gaussian measure in $\mathbb{R}^n$, and more
generally, to the family of Boltzmann measures. Moreover, as shown
in Section \ref{secsubgeneral-product}, the results extend to a
broader class of measures whose isoperimetric function satisfies
some specific condition. It seems interesting to extend the
results to broader classes of measures, and on the other hand, to
determine measures for which such results cannot be obtained.
\end{itemize}
The paper is organized as follows: in
Section \ref{secGeometric-meaning} we prove Propositions
\ref{luniformenlarge} and \ref{proprusso-location},
thus establishing the geometric meaning of the new definition. In
Section \ref{secKKL} we discuss the relation between the
geometric influences and the $h$-influences, and prove
Theorem \ref{ThmOur-KKL}. In Section \ref{secIsoperimetry} we
apply Theorem \ref{ThmOur-KKL} to establish a lower bound on the
size of the boundary of transitive sets with respect to uniform
enlargement, proving Theorem \ref{thmisolowerbound}. Finally,
in Section \ref{secrotation} we study the effect of \textit{rotations}
on the geometric influences. We conclude the introduction
with a brief statistical application of
the results established here.

\subsection{A statistical application} Let $Z_1, Z_2, \ldots, Z_n$ be i.i.d.
$\mathrm{N}(\theta, 1)$. Suppose we want to test the hypothesis
$ H_0 \dvtx\theta= \theta_0$ vs. $H_1\dvtx\theta= \theta_1$
$(\theta_1> \theta_0)$ with level of significance at most $\beta$
(for some $0<\beta<1/2$).

The remarkable classical result by Neyman and Pearson
\cite{neyman33} says that the most powerful test for the above
problem is based on the sample average $\bar Z_n = n^{-1}
\sum_{i=1}^n Z_i $, and the critical region of the test is given
by \mbox{$\mathcal C_{\mathrm{mp}} = \{ \bar Z_n > K \}$} where the
constant $K$ is chosen is such that $\prob_{\theta_0}\{ \mathcal
C_{\mathrm{mp}} \} =\beta$. It can be easily checked that to
achieve power at least $1-\beta$ for this test, we need the
parameters $\theta_0$ and $\theta_1$ to be separated by at least
$| \theta_1 -\theta_0| > C(\beta)/ \sqrt{n}$ for some appropriate
constant $C(\beta)$.

Consider the following setup where the test statistics is given
by $f(Z_1, \ldots,\allowbreak Z_n)$ where $f\dvtx\mathbb R^n \to\mathbb R$ is a
measurable function which is nondegenerate, transitive and
monotone increasing in each of its coordinates. The transitivity
of~$f$ ensures equal weight is given to each data point while
constructing the test and the monotonicity of $f$ implies that the
distribution of $f$ depends on $\theta$ in a~monotone fashion.
Note that we do not assume any smoothness property of $f$. In
general the test statistics $f(Z_1, \ldots, Z_n)$, in contrast to
the sample average which is a~sufficient statistics for this
problem, may be a~result of an ``inefficient compression'' of the
data, and we have only access to the compressed data.

In this
case the critical region would be of the form \mbox{$\mathcal C\,{=}\,\{ f(Z_1,
\ldots, Z_n)\,{>}\,K\}$} where $K$ is chosen so that $\prob_{\theta_0} \{
\mathcal{C} \} = \beta$.

Note that regions $\mathcal{C} $ satisfy:
\begin{longlist}[(iii)]
\item[(i)] $\prob_{\theta_0} \{ \mathcal{C} \} = \beta$;
\item[(ii)] $\mathcal{C}$ is transitive;
%
\item[(iii)] $\mathcal{C}$ is an increasing set.
\end{longlist}



Clearly, the most powerful test belongs to this class, but, in
general, a test of the above type can be of much less power. An
interesting open question will be to find the worst test (i.e.,
having lowest power) among all tests satisfying (i), (ii) and
(iii). Intuitively if $\theta_1$ and $\theta_0$ are far apart,
even a very weak test can detect the difference between the null
and the alternative. Corollary \ref{CorSharp-Threshold} gives us
a quantitative estimate of how far apart the parameters need to be
so that we can safely distinguish them no matter what test we use.
Indeed any test satisfying (i), (ii) and (iii) still has power of at
least $1-\beta$ as long as $|\theta_1 - \theta_0| > c
\log(1/2\beta)/\sqrt{\log n}$ for some absolute constant $c$.

For the test $\{ \max_i Z_i > K\}$, the dependence on $n$ in the above
bound is tight up to constant factors.

We briefly note that the statistical reasoning introduced here may be combined
with Theorem 2.1 in \cite{Friedgut-Kalai}. Thus a similar statement
holds when~$Z_1,\allowbreak Z_2, \ldots, Z_n$ are i.i.d.
$\operatorname{Bernoulli}(p)$, and we want to test the hypothesis
$ H_0 \dvtx\allowbreak p = p_0$ vs. $H_1\dvtx p = p_1$ $(1> p_1> p_0> 0)$. In this case,
the power of any test satisfying (i), (ii) and (iii) is at least
$1-\beta$ as long as $|p_1 - p_0| > c \log(1/2\beta)/\log n$ for
some absolute constant~$c$.\vspace*{-3pt}

%
\section{Boundary under uniform enlargement and derivatives}
\label{secGeometric-meaning}
In this section we provide the geometric
interpretation of the influence. We begin by proving
Proposition \ref{luniformenlarge}.\vspace*{-2pt}

\subsection{\texorpdfstring{Proof of Proposition \protect\ref{luniformenlarge}}{Proof of Proposition 1.3}}

In our proof we use the following simple lemma:\vadjust{\goodbreak}
\begin{Lemma} \label{ldensitybound}
Let $\lambda$ be as given in Proposition \ref{luniformenlarge}.
Given $\eps> 0$, there exists a constant $C_\eps> 0$ such that
for all $x, y \in\Real$,
\[
|\lambda(y) - \lambda(x)| \le C_\eps| \gL(y) - \gL(x) | + \eps/4.
\]
\end{Lemma}
\begin{pf}
Since $\lim_{ |z| \rightarrow\infty} \lambda(z) = 0$, there
exists $z_0>0$ such that
\[
\sup_{|z| \ge z_0} \lambda(z) \le\eps/8.
\]
Fix $x,y \in\mathbb{R}$, and assume without loss of generality
that $x \leq y$. We have
%
%
\begin{eqnarray}\label{Neweq21}
|\lambda(y)-\lambda(x)| &\leq&\biggl| \int_x^{\max(x,-z_0)} \lambda'(z)\,dz
\biggr| + \biggl| \int_{\max(x,-z_0)}^{\min(z_0,y)} \lambda'(z)\,dz
\biggr| \nonumber\\[-8pt]\\[-8pt]
&&{}+ \biggl| \int_{\min(z_0,y)}^y \lambda'(z)\,dz \biggr|.\nonumber
\end{eqnarray}
By the choice of $z_0$,
%
%
\begin{equation}\label{Neweq22}
\biggl| \int_{\min(z_0,y)}^y \lambda'(z)\,dz \biggr| \leq
|\lambda(y)-\lambda(z_0)|\leq\eps/8,
\end{equation}
and similarly,
%
%
\begin{equation}\label{Neweq23}
\biggl| \int_x^{\max(x,-z_0)} \lambda'(z)\,dz \biggr| \leq\eps/8.
\end{equation}
On the other hand, since the function $\lambda'/ \lambda$ is continuous,
there exists $C_{\eps}$ such that $|\lambda'(z)|/ \lambda(z) \le
C_{\eps}$
for all $|z| \le z_0$. Hence,
%
%
\begin{eqnarray}\label{Neweq24}
\biggl| \int_{\max(x,-z_0)}^{\min(z_0,y)} \lambda'(z)\,dz \biggr| &\leq&
C_{\eps} \int_{\max(x,-z_0)}^{\min(z_0,y)} \lambda(z)\,dz\nonumber\\
&=& C_{\eps}
\bigl(\gL(\min(z_0,y))-\gL\bigl(\max(x,-z_0)\bigr) \bigr) \\
&\leq& C_{\eps} \bigl(\gL(y) - \gL(x)\bigr) = C_{\eps} |\gL(y) - \gL(x)|.
\nonumber
\end{eqnarray}
Substitution of (\ref{Neweq22}), (\ref{Neweq23})
and (\ref{Neweq24}) into (\ref{Neweq21}) yields the
assertion.
\end{pf}

Now we are ready to present the proof of
Proposition \ref{luniformenlarge}.
\begin{pf*}{Proof of Proposition \ref{luniformenlarge}}
Without loss of generality, assume that~$A$ is decreasing. Thus, $
\nu^{\otimes n} (A+ [-r, r]^n) = \nu^{\otimes n} (A+ [0, r]^n)$.
We decompose $\nu^{\otimes n} (A+ [0, r]^n) - \nu^{\otimes n} (A)$
as a telescoping sum
%
%
\begin{equation}\label{Eq211}
\sum_{ i=1}^n \nu^{\otimes n} ( A +[0, r]^{i}\times\{0\}^{n-i}) -
\nu^{\otimes n} ( A +[0, r]^{i-1}\times\{0\}^{n-i+1}).\hspace*{-15pt}
\end{equation}
It follows immediately from (\ref{Eq211}) that it is sufficient
to show that given $\eps
> 0$, there exists $\delta> 0$ such that for all $1 \leq i \leq n$ and
for all $0< r< \delta$,
%
%
\begin{eqnarray}\label{Eq212}
&&\biggl|\frac{\nu^{\otimes n} ( A +[0,r]^{i-1}\times[0, r]\times
\{0\}^{n-i}) - \nu^{\otimes n} ( A +[0, r]^{i-1}\times
\{0\}^{n-i+1})}{r}\nonumber\hspace*{-30pt}\\
&&\hspace*{274pt}{} - \ig_i(A) \biggr|\hspace*{-30pt} \\
&&\qquad\le\eps.\nonumber\hspace*{-30pt}
\end{eqnarray}
For a fixed $i$, define
\[
B_r^i = A +[0,r]^{i-1}\times\{0\}^{n-i+1}.
\]
Obviously, $B_r^i$ is a decreasing set. Note that $ A
+[0,r]^{i-1}\times[0, r]\times\{0\}^{n-i} = B_r^i +
\{0\}^{i-1}\times[0, r]\times\{0\}^{n-i}$. Hence,
equation (\ref{Eq212}) can be rewritten as
%
%
\begin{equation}\label{Eq213}
\biggl|\frac{\nu^{\otimes n} (B_r^i + \{0\}^{i-1}\times[0,
r]\times\{0\}^{n-i}) - \nu^{\otimes n} (B_r^i)}{r} - \ig_i(A)
\biggr| \le\eps.\hspace*{-15pt}
\end{equation}
For any decreasing set $D \subset\Real^n$ and for any
$x \in\Real^n$, define
\[
t_i(D;x) := \sup\{ y \dvtx y \in D^x_i\} \in[- \infty, \infty]
\]
with the convention that the supremum of the empty set is $-\infty$.
We use two simple observations:
\begin{longlist}[(2)]
\item[(1)] For any decreasing set $D$ (and in particular, for $A$ and
for $B_r^i$), it is clear that $\nu^{\otimes n} (D) = \E_x
\gL(t_i(D;x))$.

\item[(2)] For a decreasing set $D$, we have $\ig_i(D)= \E_x
\lambda(t_i(D;x))$. This follows from a known property of the lower
Minkowski content: in the case when $\nu$ has a continuous density
$\lambda$, and $L$ is a semi-infinite ray, that is, $L = [\ell,
\infty)$ or $L = (-\infty, \ell]$, we have $\nu^{+}(L) =
\lambda(\ell)$.
\end{longlist}
We further observe that
%
%
\begin{eqnarray}\label{Eq214}
&&\biggl| \frac{\nu^{\otimes n} (B_r^i + \{0\}^{i-1}\times[0,
r]\times\{0\}^{n-i}) -\nu^{\otimes n} ( B_r^i)}{r} - \E_x
\lambda(t_i(B_r^i;x)) \biggr| \nonumber\hspace*{-30pt}\\[-8pt]\\[-8pt]
&&\qquad\leq r \|\lambda'\|_{\infty}.\nonumber\hspace*{-30pt}
\end{eqnarray}
Indeed, by observation (1), the left-hand side of (\ref{Eq214}) is equal
to
%
%
\begin{equation}\label{Eq215}
\biggl| \E_x \biggl[\frac{\gL(t_i(B_r^i;x)+r) - \gL(t_i(B_r^i;x))}{r}
- \lambda(t_i(B_r^i;x)) \biggr] \biggr|.
\end{equation}
By the mean value theorem, there exists $h \in[0,r]$ such that
\[
\frac{\gL(t_i(B_r^i;x)+r) - \gL(t_i(B_r^i;x))}{r} =
\lambda\bigl(t_i(B_r^i;x)+h\bigr),
\]
and thus
\[
\mbox{(\ref{Eq215})} = \bigl| \E_x
\bigl[\lambda\bigl(t_i(B_r^i;x)+h\bigr)-\lambda(t_i(B_r^i;x))\bigr] \bigr| \leq r
\|\lambda'\|_{\infty}.\vadjust{\goodbreak}
\]
%

Combining (\ref{Eq213}) and (\ref{Eq214}), and
ensuring that $r<\eps/(2\|\lambda'\|_{\infty})$, it is sufficient to
show that
\[
| \E_x \lambda(t_i(B_r^i;x)) - \ig_i(A) | \leq\eps/2,
\]
and by observation (2), this is equivalent to
%
%
\begin{equation}
| \E_x \lambda(t_i(B_r^i;x)) - \E_x \lambda(t_i(A;x)) | \leq
\eps/2.
\end{equation}
By Lemma \ref{ldensitybound} and observation (1), we have
\begin{eqnarray*}
| \E_x \lambda(t_i(B_r^i;x)) - \E_x \lambda(t_i(A;x)) | &\le& C_\eps
\E_x |
\gL(t_i(B_r^i;x)) - \gL(t_i(A;x))| +\eps/4 \\
&=& C_\eps\E_x \bigl( \gL(t_i(B_r^i;x)) - \gL(t_i(A;x)) \bigr) +\eps/4 \\
&=& C_\eps\bigl({\nu^{\otimes n} }(B_r^i) - {\nu^{\otimes n} }(A)\bigr)
+\eps/4.
\end{eqnarray*}
It thus remains to show that there exists $\delta> 0$
sufficiently small such that for all $ 0 < r < \delta$,
%
%
\begin{equation}\label{Eq216}
{\nu^{\otimes n} }(B_r^i) - {\nu^{\otimes n} }(A) \leq
\frac{\eps}{4C_\eps}.
\end{equation}
We can write
\begin{eqnarray*}
&&{\nu^{\otimes n} }(B_r^i) - {\nu^{\otimes n} }(A) \\
&&\qquad=
\sum_{j=1}^{i-1} \bigl({\nu^{\otimes n} }(A+[0,r]^j \times
\{0\}^{n-j}) - {\nu^{\otimes n} }(A+[0,r]^{j-1} \times
\{0\}^{n-j+1}) \bigr),
\end{eqnarray*}
and thus it is sufficient to find $\delta> 0$ such that for all
$0< r< \delta$ and for all $1 \leq j \leq i-1$,
\[
{\nu^{\otimes n} }(A+[0,r]^j \times\{0\}^{n-j}) - {\nu^{\otimes
n} }(A+[0,r]^{j-1} \times\{0\}^{n-j+1}) \le
\frac{\eps}{4nC_\eps}.
\]
Since for any decreasing $D \subset\Real^n$,
\[
|{\nu^{\otimes n} } (D+ \{0\}^{j-1}\times[0,r] \times\{0\}^{n-j}) -
{\nu^{\otimes n} }(D)| \le\|\lambda\|_\infty r,
\]
we can choose $ \delta= \min\{\frac{\eps}{4nC_\eps
\|\lambda\|_\infty},\frac{\eps}{2\|\lambda'\|_{\infty}}\}$. This completes
the proof.
\end{pf*}
\begin{remark}\label{remconvex}
We note that the same proof (with minor modifications) holds for
any convex set $A$. The only nonobvious change is noting that the
Minkowski content of a segment $[a,b]$ is $\nu^{+}([a,b]) =
\lambda(a)+\lambda(b)$, where $\lambda$ is the density of the
measure $\nu$.
On the other hand, it is clear that the statement of
Proposition~\ref{luniformenlarge} does not hold for general
measurable sets. For example, if $A=\mathbb{Q}^n$ where
$\mathbb{Q}$ is the set of rational numbers, then the size of the
boundary of $A$ with respect to a uniform enlargement is $\infty$,
while the sum of geometric influences of $A$ is zero. It seems an
interesting question to determine to which classes of measurable
sets Proposition \ref{luniformenlarge} applies.
\end{remark}




\subsection{\texorpdfstring{Proof of Proposition \protect\ref{proprusso-location}}{Proof of Proposition 1.6}}
Define a function
$\Pi\dvtx\Real^n \to[0, \infty)$ by
\[
\Pi(\ga_1, \ldots, \ga_n) =
\nu_{\ga_1} \otimes\cdots\otimes\nu_{\ga_n} (A).\vadjust{\goodbreak}
\]
The partial derivative of $\Pi$ with respect to the $i$th
coordinate can be written as
%
%
\begin{equation}\label{Eq221}
\frac{\partial\Pi(\ga_1, \ldots, \ga_n )}{\partial\ga_i} =
\lim_{r \downarrow0} \frac{\E_x \nu_{\ga_i+r} (A^x_i) - \E_x
\nu_{\ga_i} (A^x_i)}{r}.
\end{equation}
For $x \in\Real^n$, define
\[
s_i(A;x) := \inf\{ y \dvtx y \in A^x_i\} \in[- \infty, \infty].
\]
Since $A$ is monotone increasing, for any $x \in\Real^n$ we have
%
%
\begin{eqnarray}\label{Eq222}\qquad
\frac{\nu_{\ga_i+r} (A^x_i) - \nu_{\ga_i} (A^x_i)}{r} &=& \frac{
\nu_{\ga_i+r}( [s_i(A;x), \infty) ) - \nu_{\ga_i}( [s_i(A;x),
\infty) ) }{r} \nonumber\\[-8pt]\\[-8pt]
&=& \frac1r \int_{s_i(A;x)-r}^{s_i(A;x)}
\lambda_{\ga_i}(z)\,dz,\nonumber
\end{eqnarray}
and by the fundamental theorem of calculus, this expression
converges to $\lambda_{\ga_i}(s_i(A;x))$ as $ r \to0$. Moreover,
(\ref{Eq222}) is uniformly bounded by $\| \lambda_{\ga_i}\|_\infty=
\| \lambda\|_\infty$ (which is finite since $\lambda$ is bounded by the
hypothesis). Therefore, by the dominated convergence theorem, it
follows that the first-order partial derivatives of $\Pi$ exist
and are given by
\[
\frac{\partial\Pi(\ga_1, \ldots, \ga_n )}{\partial\ga_i} = \E_{
x \sim\nu_{\ga_1} \otimes\cdots\otimes\nu_{\ga_n}}
\lambda_{\ga_i}(s_i(A;x)) = \ig_i(A),
\]
where the influence is with respect to the measure $\nu_{\ga_1}
\otimes\cdots\otimes\nu_{\ga_n}$. [For the last equality, see
observation (2) in the proof of
Proposition \ref{luniformenlarge} above. Here we use the
convention that $ \lambda_{\ga_i}(-\infty) = \lambda_{\ga
_i}(\infty) =
0$.]

Hence, by the chain rule, it is sufficient to check that all the
partial derivatives of $\Pi$ are continuous at $(\ga, \ldots,
\ga)$. Without loss of generality, we assume that $\ga=0$. Note
that
%
%
\begin{equation} \label{eqrussofirst}
\E_{ x \sim\nu_{\ga_1} \otimes\cdots\otimes\nu_{\ga_n}}
\lambda_{\ga_i}(s_i(A;x))
= \E_{ x \sim\nu\otimes\cdots\otimes\nu}
\Biggl(\prod_{j=1}^n \frac{\lambda_{\ga_j}(x_j)}{\lambda(x_j)}\Biggr) \lambda
_{\ga_i}(s_i(A;x)).\hspace*{-40pt}
\end{equation}
For each $x \in\Real^n$,
%
%
\begin{equation} \label{Eq224}
\prod_{j=1}^n \frac{\lambda_{\ga_j}(x_j)}{\lambda(x_j)}
\lambda_{\ga_i}(s_i(A;x)) \to\prod_{j=1}^n \lambda(s_i(A;x))
\end{equation}
as $\max{|\ga_i|} \to0$. Hence, the continuity of the partial
derivatives would follow from the dominated convergence theorem
if (\ref{Eq224}) was uniformly bounded. In order to obtain such
bound, we consider a compact subset.

There exist $\kappa_L < K_L < K_R < \kappa_R$ and $\gd> 0$ such
that $\nu( [K_L+ \gd, K_R- \gd]) \ge1 - \eps$. Let $c:= \min_{ z
\in[K_L, K_R] } \lambda(z)$. Note that by the hypothesis on $\lambda
$, we
have $c>0$. If $|\ga_j| \le\gd$ for all $j$, then
%
%
\begin{eqnarray} \label{Eq225}
&&\Biggl|\mbox{(\ref{eqrussofirst})} - \E_{ x \sim\nu\otimes\cdots
\otimes\nu} \Biggl(\prod_{j=1}^n \frac{\lambda_{\ga_j}(x_j)}{\lambda(x_j)}
1_{ \{K_L \le x_j \le K_R\} }\Biggr) \lambda_{\ga_i}(s_i(A;x))
\Biggr|\nonumber\hspace*{-30pt}\\[-8pt]\\[-8pt]
&&\qquad
\le\eps\cdot n \cdot\| \lambda\|_\infty.\nonumber\hspace*{-30pt}
\end{eqnarray}
Indeed, denoting $S=\{x \in\Real^n\dvtx\exists j, x_j \notin[K_L,
K_R]\}$ and using (\ref{eqrussofirst}), we have
\[
\mbox{(\ref{Eq225})}\!=\!| \E_{ x \sim\nu_{\ga_1} \otimes\cdots
\otimes\nu_{\ga_n}} 1_S \lambda_{\ga_i}(s_i(A;x)) |\!\leq\!
\|\lambda\|_{\infty} \E_{ x \sim\nu_{\ga_1} \otimes\cdots
\otimes
\nu_{\ga_n}} 1_S\!\leq\!\eps\cdot n \cdot\| \lambda\|_\infty,
\]
where the last inequality is a union bound using the choice of
$K_L$ and $K_R$.

Similarly, by a union bound we have
%
%
\begin{eqnarray} \label{Eq226}
&&\bigl| \E_{ x \sim\nu\otimes\cdots\otimes\nu} \lambda(s_i(A;x))
- \E_{ x \sim\nu\otimes\cdots\otimes\nu} 1_{ \{K_L \le x_j
\le K_R \ \forall j\} }\lambda(s_i(A;x))\bigr| \nonumber\hspace*{-30pt}\\[-8pt]\\[-8pt]
&&\qquad\le\eps\cdot n \cdot
\| \lambda\|_\infty.\nonumber\hspace*{-30pt}
\end{eqnarray}
Combining (\ref{Eq225}) with (\ref{Eq226}), it is sufficient
to prove that
\begin{eqnarray*}
&&\E_{ x \sim\nu\otimes\cdots\otimes\nu} \prod_{j=1}^n
\frac{\lambda_{\ga_j}(x_j)}{\lambda(x_j)} 1_{ \{K_L \le x_j \le
K_R\} }
\lambda_{\ga_i}(s_i(A;x))
\\
&&\qquad\to\E_{ x \sim\nu\otimes\cdots\otimes
\nu} \prod_{j=1}^n 1_{ \{K_L \le x_j \le K_R\} } \lambda(s_i(A;x)).
\end{eqnarray*}
This indeed follows from the dominated convergence theorem, since
for each \mbox{$x \in\Real^n$},
\[
\prod_{j=1}^n \frac{\lambda_{\ga_j}(x_j)}{\lambda(x_j)}
1_{ \{K_L \le x_j \le K_R\} } \lambda_{\ga_i}(s_i(A;x))
\to
\prod_{j=1}^n 1_{ \{K_L \le x_j \le K_R\} } \lambda(s_i(A;x))
\]
as $\max{|\ga_i|} \to0$ and is uniformly bounded by $c^{-n}
\|\lambda\|^{n+1}_\infty$. This completes the proof.

\section{Relation to $h$-influences and a general lower bound on geometric influences}
\label{secKKL}

In this section we analyze the geometric influences by reduction
to problems concerning $h$-influences introduced in a recent paper
by the first author \cite{keller09b}. First we describe and extend
the results on $h$-influences, and then we show their relation to
geometric influences.

\subsection{$h$-influences}

\begin{definition}
Let $h\dvtx[0,1] \to[0,\infty)$ be a measurable function. For a
measurable subset $A$ of $X^n$ equipped with a product measure
$\nu^{\otimes n}$, the $h$-influence of the $i$th coordinate on
$A$ is
\[
I^h_i(A) := \E_{ x} [h(\nu( A^x_i ))].
\]
\end{definition}

The two main results concerning $h$-influences are a
monotonization lemma and an analog of the KKL theorem.
\begin{Lemma}[(\cite{keller09b})] \label{theomono}
Consider the space $[0,1]^n$, endowed with the product Lebesgue
measure $u^{\otimes n}$. Let $h\dvtx[0, 1] \to[0, 1]$ be a concave
continuous function. For every Borel measurable set $A \subseteq
[0,1]^n$, there exists a monotone increasing set $B \subseteq
[0,1]^n$ such that:
\begin{longlist}[(2)]
\item[(1)] $u^{\otimes n}(A) = u^{\otimes n}(B)$;
\item[(2)] for
all $1\le i \le n$, we have $I^h_i(A) \ge I^h_i(B)$.
\end{longlist}
\end{Lemma}
\begin{Theorem}[(\cite{keller09b})]\label{theohinf}
$\!\!$Denote the entropy function as $\Ent(x) := -x \log x -(1-x) \log
(1-x)$ for all $0<x<1$, and $\Ent(0)=\Ent(1)=0$. Consider the
space $[0,1]^n$, endowed with the product Lebesgue measure
$u^{\otimes n}$. Let $h\dvtx[0,1] \to[0,1]$ such that $h(x) \ge
\Ent(x)$ for all $0 \le x\le1$. Then for every measurable set $ A
\subseteq[0,1]^n$ with $u^{\otimes n}(A) = t$, there exists $1
\le i \le n$ such that the $h$-influence of the $i$th coordinate
on $A$ satisfies
\[
I^h_i(A) \ge ct(1-t) \log n/ n,
\]
where $c>0$ is a universal constant.
\end{Theorem}

Other results on $h$-influences which we shall use later include
analogs of several theorems concerning influences on the
discrete cube: Talagrand's lower bound on the vector of
influences \cite{Talagrand1}, a variant of the KKL theorem for
functions with low influences \cite{Friedgut-Kalai} and
Friedgut's theorem asserting that a~function with a~low influence
sum essentially depends on a few coordinates~\cite{Friedgut1}.

In the application to geometric influences we would like to use
$h$-influences for certain functions $h$ that do not dominate the
entropy function. In order to overcome this problem, we use the
following lemma that allows to relate general $h$-influences to
the entropy-influence [i.e., $h$-influence for $h(x)=\Ent(x)$].

\begin{Lemma}\label{theomain}
Consider the product space $(\mathbb R^n, \nu^{\otimes n})$, where
$\nu$ has a continuous cumulative distribution function $\gL$.
Let $h\dvtx[0,1] \to[0, \infty)$, and let $A \subseteq\Real^n$ be a
Borel-measurable set. For all $1 \leq i \leq n$,
%
%
\begin{equation}\label{EqEntropy-General}
I^h_i(A) \geq\tfrac{1}{2} \delta\cdot I^{\Ent}_i(A),
\end{equation}
where
%
%
\begin{equation} \label{eqratio}
\gd=\gd(A,i) =\inf_{x \in[\vartheta(I^{\Ent}_i(A)/2), 1 -
\vartheta(I^{\Ent}_i(A)/2)]} \frac{h(x)}{ \Ent(x)}
\end{equation}
and $\vartheta(y) = y/(-2 \log y)$.
\end{Lemma}

\begin{pf}
Set $f = 1_A$. Let $u$ be the Lebesgue measure on $[0,1]$.
Define $ g(x_1, \ldots, x_n) := f(\gL^{-1}(x_1), \ldots,
\gL^{-1}(x_n))$ and write $B$ for\vadjust{\goodbreak} the set $\{x \in\Real^n\dvtx\allowbreak g(x)
=1\}$. Since $\gL^{-1}(u) \stackrel{d}{=} \nu$, the set $B$
satisfies $u^{\otimes n}(B) = \nu^{\otimes n} (A) = t$ and
\[
I^h_i(B)|_{u^{\otimes n}} = I^h_i(A) |_{\nu^{\otimes n}}
\qquad\mbox{for each } 1 \le i \le n.
\]
Denote by $\alpha$ the unique value in the segment
$[0,1/2]$ which satisfies the equation
$\alpha=\Ent^{-1}(I^{\Ent}_i(A)/2)$. It is clear that for any $x
\notin[\alpha,1-\alpha]$,
\[
\Ent(x) \leq\Ent\bigl(\Ent^{-1}\bigl(I^{\Ent}_i(A)/2\bigr) \bigr) =
I^{\Ent}_i(A)/2,
\]
and thus,
\begin{eqnarray*}
\E_{x} \bigl[\Ent( u( B^x_i) ) 1_{\{ u(B^x_i) \in
[\alpha,1-\alpha]\} } \bigr] &=& I^{\Ent}_i(B)|_{u^{\otimes n}}
- \E_{x} \bigl[ \Ent( u(
B^x_i) ) 1_{\{ u(B^x_i) \notin[\alpha,1-\alpha] \} } \bigr] \\
&\geq& I^{\Ent}_i(A)/2.
\end{eqnarray*}
Therefore, by (\ref{eqratio}),
\begin{eqnarray*}
I^h_i(A) |_{\nu^{\otimes n}} &=& I^h_i(B)|_{u^{\otimes n}}
\ge\E_{x} \bigl[ h ( u( B^x_i) ) 1_{\{ u(B^x_i) \in
[\alpha,1-\alpha] \} } \bigr]\\
&\geq& \biggl(\inf_{x \in[\Ent^{-1}(I^{\Ent}_i(A)/2), 1 -
\Ent^{-1}(I^{\Ent}_i(A)/2)]} \frac{h(x)}{ \Ent(x)} \biggr) I^{\Ent
}_i(A)/2\\
&\ge&\delta\cdot I^{\Ent}_i(A)/2,
\end{eqnarray*}
where the last step follows from the fact that $\vartheta(x) \le\Ent
^{-1}(x)$ for $x \le1/2$ which is easy to verify.
\end{pf}



\subsection{Relation between geometric influences and $h$-influences for log-concave measures}

It is straightforward to check the following relation between the
geometric influences and the $h$-influences for monotone sets. The
proof follows immediately from observation (2) in the proof of
Proposition \ref{luniformenlarge}.
\begin{Lemma}\label{lemhgeoinflu}
Consider the product space $(\mathbb R^n, \nu^{\otimes n})$ where
$\nu$ has a continuous density $\lambda$. Let $\gL$ denote the
cumulative distribution function of $\nu$. Then for any monotone
set $A \subseteq\Real^n$,
\[
\ig_i(A) = I^h_i(A) \qquad\forall1 \le i \le n,
\]
where $h(t) = \lambda( \gL^{-1}(t))$ when $A$ is decreasing and
$h(t)= \lambda( \gL^{-1}(1-t))$ when~$A$ is increasing. Here $\gL
^{-1}$ denotes the unique inverse of the function $\gL$.
\end{Lemma}

Using Lemmas \ref{theomono} and \ref{lemhgeoinflu}, we
can obtain a monotonization lemma for geometric influences that
holds if the underlying measure has a log-concave density. In
order to show this, we use the following isoperimetric inequality
satisfied by log-concave distributions (see, e.g.,
\cite{bobkov96}).
\begin{Theorem}[(\cite{bobkov96})]\label{ThmIso-Bobkov}
Let $\nu$ have a log-concave density $\lambda$, and let $\gL$ be the
corresponding cumulative distribution function. Denote the
(unique) inverse of the function $\gL$ by $\gL^{-1}$. Fix any\vadjust{\goodbreak} $t
\in(0,1)$; that is,
for $t \in(0, 1)$ and for every Borel-measurable set $A \subseteq
R$ with $\nu(A)=t$,
%
%
\begin{eqnarray}\label{inebobkov}
\nu(A+[-r, r]) \ge\min\bigl\{ \gL\bigl( \gL^{-1}(t) + r \bigr), 1- \gL\bigl(
\gL^{-1}(1-t) - r \bigr) \bigr\}\nonumber\\[-8pt]\\[-8pt]
&&\eqntext{\forall r >0.}
\end{eqnarray}
Moreover, in the class of all Borel-measurable sets of
$\nu$-measure $t$, the extremal sets, that is, the sets for which
(\ref{inebobkov}) holds as an equality, are intervals of the form
$(-\infty, a]$ or $[a, \infty)$ for some $a \in\Real$.

If $\lambda$ is symmetric (around the median), then inequality (\ref
{inebobkov}) is simplified to
%
%
\begin{equation}\label{eqisolog-concave}
\nu(A+[-r, r]) \ge\gL\bigl( \gL^{-1}(t) + r \bigr) \qquad\forall r >0.
\end{equation}
\end{Theorem}

Now we are ready to present the monotonization lemma.
\begin{Lemma}\label{lemincreasingequality}
Consider the product measure $\nu^{\otimes n}$ on $\Real^n$ where
$\nu$ is a~probability distribution with a continuous symmetric
log-concave density~$\lambda$ satisfying $\lim_{ |z| \rightarrow
\infty} \lambda(z) = 0$. Then for any Borel set $A \subset
\Real^n$:
\begin{longlist}[(ii)]
\item[(i)] $\ig_i(A) \ge I_i^h(A)$ for all $1 \le i \le n$,
where $h(t) =\lambda( \gL^{-1}(t))$;

\item[(ii)] there exists an increasing set $B$ such that
$\nu^{\otimes n}(B)= \nu^{\otimes n}(A)$ and
\[
\ig_i(B) \le\ig_i(A) \qquad\mbox{for all } 1 \le i \le n.
\]
\end{longlist}
\end{Lemma}
\begin{pf}
Let $\gL$ be the cumulative distribution of $\nu$. Fix $x \in
\mathbb R^n$. By Theorem~\ref{ThmIso-Bobkov}, we have, for all $r
> 0$,
\[
\frac{ \nu(A^x_i+[-r, r]) - \nu(A^x_i)}{r} \ge\frac{\gL(
\gL^{-1}( \nu(A^x_i)) + r ) - \gL(\gL^{-1}(\nu(A^x_i)))}{r}.
\]
Taking limit of the both sides as $r \downarrow0$, we obtain
\[
\nu^{+}(A^x_i) \ge\lambda(\gL^{-1}(\nu(A^x_i)) ) = h(\nu(A^x_i)),
\]
which implies the first part of the lemma.

For a proof of the second part, we start by noting that the
assumptions on~$\nu$ imply that $h$ is concave and continuous.
Thus we can invoke Lemma \ref{theomono} to find an increasing
set $B$ such that $\nu^{\otimes n}(B)= \nu^{\otimes n}(A)$ and $
I^h_i(B) \le I^h_i(A)$ for all $1 \le i \le n$. By the first part
of the lemma, $I_i^h(A) \le\ig_i(A)$ for all $1 \le i \le n$. On
the other hand, it follows from Lemma \ref{lemhgeoinflu} that
$\ig_i(B) = I^h_i(B)$ for all $1 \le i \le n$. Hence,
\[
\ig_i(B) = I^h_i(B) \le I^h_i(A) \le\ig_i(A)
\]
as asserted.
\end{pf}

To keep our exposition simple, we will restrict our attention to
an important family of log-concave distributions known as
Boltzmann measures for the rest of the section. We mention in
passing that some of the techniques that we are going to develop
can be applied to other log-concave measures with suitable
isoperimetric properties.

\subsection{Lower bounds on geometric influences for Boltzmann measures}

%
\begin{definition}[(Boltzmann measure)] \label{defboltzmann}
The density of the Boltzmann measure $\mu_{\rho}$ with parameter
$\rho\ge1$ is given by
\[
\phi_{\rho}(x) := \frac{1}{2 \Gamma(1+1/\rho)} e^{ - |x|^{\rho
}}\,dx,\qquad x \in\mathbb R.
\]
\end{definition}

Note that $\rho=2$ corresponds to the Gaussian measure with
variance $1/2$ while $\rho=1$ gives the two-sided exponential
measure.

We have the following estimates on the tail probability of Boltzmann
measures.
\begin{Lemma} \label{lemboltzmanntail} Let $\Phi_\rho$ denote the
cumulative distribution function of the Boltzmann distribution
with parameter $\rho$. Then for $ z> 0$, we have
%
\[
\frac{1}{2 \rho\Gamma(1+1/\rho)} \biggl( \frac{z}{({\rho-1})/{\rho
} +z^\rho}\biggr)e^{ - z^{\rho}} \le1- \Phi_\rho(z) \le\frac{1}{2\rho
\Gamma(1+1/\rho)} \frac{1}{ z^{\rho-1}} e^{ - z^{\rho}}.
\]
In particular,
%
%
\begin{equation} \label{eqhest}
\phi_\rho(\Phi_\rho^{-1}(x)) \asymp x(1-x) \bigl(- {\log}\bigl( x(1-x)\bigr)
\bigr)^{(\rho-1)/\rho}
\end{equation}
for $x$ close to zero or one.
\end{Lemma}
\begin{pf} Set $Z_\rho= 2 \Gamma(1+1/\rho)$. For the upper bound,
note that
\[
Z_\rho\bigl(1- \Phi_\rho(z)\bigr) = \int_z^\infty e^{-t^\rho} \,dt \le\frac
{1}{\rho z^{\rho-1} } \int_z^\infty\rho t ^{\rho-1} e^{-t^\rho} \,dt
\le\frac{1}{ \rho z^{\rho-1} } e^{ - z^\rho}.
\]
On the other hand, the lower bound is derived as follows:
\begin{eqnarray*}
Z_\rho\biggl(1+ \frac{\rho-1}{\rho z^\rho} \biggr) \bigl(1- \Phi_\rho(z)\bigr) &\ge&
\int_z^\infty\biggl(1+ \frac{\rho-1}{\rho t^\rho} \biggr) e^{-t^\rho} \,dt \\
&=& -
\frac{e^{-t^\rho} }{\rho t^{\rho-1}} \bigg|_{z}^\infty
= \frac{e^{-z^\rho} }{\rho z^{\rho-1}}.
\end{eqnarray*}
\upqed\end{pf}

It follows from Lemma \ref{lemincreasingequality}(i) and Lemma
\ref{lemboltzmanntail} that for Boltzmann measures, the
geometric influences lie between previously studied
$h$-influences. On the one hand, they are greater than
variance-influences [i.e., $h$-influences with $h(t)=t(1-t)$],
that were studied in, for example, \cite{hatami09,mossel09}. On the
other hand, for monotone sets they are smaller than the
entropy-influences.

It is well known that there is no analog of the KKL influence
bound for the variance-influence, and a tight lower bound on the
maximal variance-influence is the trivial bound
\[
\max_{1 \leq i \leq n} I_i^{\Var}(A) \geq ct(1-t)/n,
\]
where $t$ is the measure of the set $A$. This inequality is an
immediate corollary of the Efron--Stein inequality (see,
e.g., \cite{steele86}),\vadjust{\goodbreak} and the tightness is shown by the standard
example of one-sided boxes (considered in
Section~\ref{secsubtightness} below). On the other hand, the
analog of the KKL bound proved in \cite{keller09b} holds only
for $h$-influences with $h(t) \geq\Ent(t)$. In order to show
KKL-type lower bounds for geometric influences, we use the
following two results.

The first result is a dimension-free isoperimetric inequality for
the Boltzmann measures.
\begin{Lemma}[(\cite{barthe04})]\label{lembarthe}
Fix $\rho>1$, and let $\mu_\rho$ denote the Boltzmann measure with
parameter $\rho$. Then there exists a constant $k = k(\rho)>0$
such that for any $n \ge1$ and any measurable $A \in\mathbb
R^n$, we have
\[
\mu_\rho^{\otimes n} (A + [-r, r]^n) \geq
\mu_\rho\bigl\{ \bigl(-\infty, \Phi_\rho^{-1}(t)+k r\bigr]\bigr\},
\qquad t=\mu_\rho^{\otimes n} (A).
\]
\end{Lemma}

The second key ingredient is a simple corollary of
Lemma \ref{theomain}.
\begin{Lemma}\label{Lemmain}
Consider the product spaces $(\mathbb R^n, \mu_\rho^{\otimes n})$,
where $\mu_\rho$ denotes the Boltzmann measure with parameter
$\rho> 1$. For any $A \subset\mathbb R^n$ and for all $1 \leq i
\leq n$,
\[
\ig_i(A) \geq c I^{\Ent}_i(A) (-{\log}(I^{\Ent}_i(A))
)^{-1/\rho},
\]
where $c = c(\rho)>0$ is a universal constant.
\end{Lemma}
\begin{pf}
In view of Lemma \ref{lemincreasingequality}, it is sufficient
to prove that
\[
I^h_i(A) \geq c I^{\Ent}_i(A) (-{\log}(I^{\Ent}_i(A))
)^{-1/\rho}
\]
for $h(x) := \phi_\rho(\Phi_\rho^{-1}(x))$. This indeed follows
immediately from Lemma \ref{theomain} using the estimate on
$h(x)$ given in (\ref{eqhest}).
\end{pf}

Now we are ready to prove the KKL-type lower bounds. We start with
an analog of the KKL theorem \cite{kahn88}.
\begin{Theorem}\label{corlogcocavehinf}
Consider the product spaces $(\mathbb R^n, \mu_\rho^{\otimes n})$,
where $\mu_\rho$ denotes the Boltzmann measure with parameter
$\rho> 1$. There exists a constant $c =c(\rho)> 0$ such that for
all $n \ge1$ and for any Borel-measurable set $A \subset\Real^n$
with \mbox{$\nu^{\otimes n}(A)=t$}, we have
\[
\max_{1 \le i \le n} \ig_i(A) \ge ct(1-t) \frac{ (\log n)^{ 1-
1/\rho}}{n}.
\]
\end{Theorem}
\begin{pf}
The proof is divided into two cases, according to $\nu^{\otimes
n}(A)=t$. If $t(1-t)$ is not very small, the proof uses
Lemmas \ref{lemincreasingequality} and \ref{Lemmain}. If
$t(1-t)$ is very small, the proof relies on
Lemmas \ref{lemincreasingequality} and \ref{lembarthe}.


\textit{Case} A: $t(1-t) > n^{-1}$. By
Theorem \ref{theohinf}, there exists $1 \leq i \leq n$, such
that
\[
I^{\Ent}_i(A) \geq ct(1-t) \frac{\log n}{n}.\vadjust{\goodbreak}
\]
Since $t(1-t)>1/n$, it follows from Lemma \ref{Lemmain} that
\[
\ig_i(A) \geq c I^{\Ent}_i(A) (-{\log}(I^{\Ent}_i(A))
)^{-1/\rho} \geq c't(1-t) \frac{\log n}{n} \cdot(\log
n)^{-1/\rho},
\]
where $c'$ is a universal constant, as asserted.

\textit{Case} B: $t(1-t) \le n^{-1}$. In view of
Lemma \ref{lemincreasingequality}, we can assume without loss of
generality that the set $A$ is increasing. In that case, by
Proposition~\ref{luniformenlarge}, we have
\[
\sum_{i=1}^n \ig_i(A) = \liminf_{ r \downarrow
0}\frac{\mu_\rho^{\otimes n} (A + [-r, r]^n) - \mu_\rho^{\otimes
n} (A ) }{r}.
\]
By Lemma \ref{lembarthe},
%
%
\begin{equation}\label{inqbarthe}
\liminf_{ r \downarrow0}\frac{\mu_\rho^{\otimes n} (A + [-r,
r]^n) - \mu_\rho^{\otimes n} (A ) }{r} \ge k
\phi_\rho(\Phi_\rho^{-1}(t)).
\end{equation}
Since in this case $t(1-t) \leq n^{-1}$, it follows from
Lemma \ref{lemboltzmanntail} that
\[
\sum_{i=1}^n \ig_i(A) \ge k \phi_\rho(\Phi_\rho^{-1}(t)) \ge k'
t(1-t) (\log n)^{(\rho-1)/\rho}
\]
for some constant $k'(\rho)>0$. This completes the proof.
\end{pf}

Theorem \ref{ThmOur-KKL} is an immediate consequence of
Theorem \ref{corlogcocavehinf}. The derivation of
Corollary \ref{CorSharp-Threshold} from Theorem \ref{ThmOur-KKL}
and Proposition \ref{proprusso-location} is exactly the same as
the proof of Theorem 2.1 in \cite{Friedgut-Kalai} (which is the
analogous result for Bernoulli measures on the discrete cube), and
thus is omitted here.

We conclude this section with several analogs of results for
influences on the discrete cube. In the theorem below, part (1)
corresponds to Talagrand's lower bound on the vector of
influences \cite{Talagrand1}, part (2) corresponds to a variant of
the KKL theorem for functions with low influences established
in~\cite{Friedgut-Kalai}, part (3) corresponds to Friedgut's
characterization of functions with a~low influence sum \cite{Friedgut1}
and part (4) corresponds to Hatami's characterization of
functions with
a low influence sum in the continuous case~\cite{hatami09}. Statements~(1),
(3) and (4) of the theorem follow immediately using
Lemma~\ref{Lemmain} from the corresponding statements for the
Entropy-influence proved in~\cite{keller09b}, and statement (2) is an
immediate corollary of
statement~(1).
\begin{Theorem}\label{ThmTalagrand}
Consider the product spaces $(\mathbb R^n, \mu_\rho^{\otimes n})$,
where $\mu_{\rho}$ denotes the Boltzmann measure with parameter
$\rho> 1$. For all $n \ge1$, for any Borel-measurable set $A
\subset\Real^n$, and for all $\alpha>0$, we have:
\begin{longlist}[(2)]
\item[(1)] if $\mu_{\rho}^{\otimes n}(A)=t$, then
\[
\sum_{i=1}^n \frac{\ig_i(A)}{( - \log\ig_i(A))^{1-1/\rho}} \ge
c_1 t(1-t);
\]

\item[(2)] if $\mu_{\rho}^{\otimes n}(A)=t$ and $\max_{1 \leq i \leq n}
\ig_i(A) \leq\alpha$, then
\[
\sum_{i=1}^n \ig_i(A) \geq c_1 t(1-t) (-{\log}\alpha)^{1-1/\rho};
\]

\item[(3)] if $A$ is monotone and $\sum_{i=1}^n \ig_i(A) (-{\log}\ig
_i(A))^{1/\rho}=s$,
then there exists a set $B \subset\mathbb{R}^n$ such that $1_B$ is
determined by at
most $\exp(c_2s/\epsilon)$ coordinates and
$\mu_{\rho}^{\otimes n}(A \bigtriangleup B) \leq
\epsilon$;

\item[(4)] if $\sum_{i=1}^n \ig_i(A) (-{\log}\ig_i(A))^{1/\rho}=s$,
then there exists a set $B \subset\mathbb{R}^n$ such that $1_B$ can be
represented\vspace*{1pt} by a decision tree of depth at most $\exp(c_3
s/\epsilon^2)$ and $\mu_{\rho}^{\otimes n}(A \bigtriangleup B)
\leq\epsilon$,\setcounter{footnote}{3}\footnote{See, for example,
\cite{hatami09} for the definition of a decision tree.}
\end{longlist}
where $c_1,c_2$, and $c_3$ are positive constants which depend only on
$\rho$.
\end{Theorem}

Theorem \ref{ThmTalagrand-intro} is a special case of statements (1)
and (3)
of Theorem \ref{ThmTalagrand} obtained for $\rho=2$.

\subsection{A remark on geometric influences for more general product measures}\label{secsubgeneral-product}

It is worth mentioning that variants of
Theorems \ref{corlogcocavehinf} and \ref{ThmTalagrand} hold
for any measure $\nu$ on $\mathbb R$ which satisfies the following
two conditions:
\begin{itemize}
\item$\nu$ is absolutely continuous with respect to the Lebesgue
measure;

\item there exist constants $\rho\ge1, a>0$, such that for the
isoperimetric function~$\mathcal I_\nu$ of $\nu$ we have
\[
\mathcal I_\nu(t) \ge a \min(t, 1-t) \bigl( - \log\min(t, 1-t) \bigr)^{1
- 1/\rho},\qquad t \in[0, 1].
\]
\end{itemize}
The proofs are similar to those given for Boltzmann measures,
except for the following changes:
\begin{itemize}
\item Lemma \ref{lemincreasingequality}(i) now holds with $h(t)
= \mathcal I_\nu(t)$;

\item Lemma \ref{lemincreasingequality}(ii) does not hold in
general, but this is not a problem since for the proof of
Theorem \ref{corlogcocavehinf} we only need the first part of
the lemma;

\item instead of Lemma \ref{lembarthe}, we use the following
dimension-free isoperimetric inequality which holds for the
product measure $\nu^{ \otimes n}$ (see \cite{barthe04}):

For all $n \ge1$ and for any measurable set $A
\subseteq\mathbb R^n$,
\begin{eqnarray*}
&&
\liminf_{r \downarrow0} \frac{\nu^{ \otimes n}( A +[-r, r]^n)
- \nu^{ \otimes n}( A)}{r} \\
&&\qquad\ge\frac{a}{K}
\min(t, 1-t) \bigl( - \log\min(t, 1-t) \bigr)^{1 - 1/\rho},
\end{eqnarray*}
where $ t = \nu^{ \otimes n}(A)$ and $K>0$ is a universal constant.
\end{itemize}




\section{Boundaries of transitive sets under uniform enlargement}
\label{secIsoperimetry}

In the Gaussian space, the isoperimetric inequality for uniform
enlargement follows from the classical Gaussian isoperimetric
inequality by
Sudakov and Tsirelson~\cite{sudakov78}, Borell \cite{Borell75}
(see also \cite{ehrhard83,bobkov96,bakry96}) and the fact that
the boundary of a set under uniform enlargement always dominates its
usual boundary (i.e., the boundary under $L^2$ enlargement). To be specific,
the boundary under uniform enlargement of any measurable set $A
\subset\Real^n$ with $\mu^{\otimes n}(A)=t$, where $\mu$ is the
Gaussian measure on $\mathbb{R}$, obeys the following lower bound:
%
%
\begin{equation}\label{EqBobkov}
\liminf_{r \downarrow0}\frac{\mu^{\otimes n}(A+[-r, r]^n) -
\mu^{\otimes n}(A)}{r} \ge\phi(\Phi^{-1}(t))
\end{equation}
(where $\Phi$ and $\phi$ are the cumulative distribution function
and the density of the Gaussian distribution in $\Real$), and it is
easy to check that the
bound is achieved when $A$ is an ``axis-parallel'' halfspace (i.e., sets
of the form $\{x \in\mathbb R^n\dvtx x_i \le a \}$ or its complement) with
$\mu^{\otimes n}(A) =t$.


In this section we consider the same isoperimetric problem under
an additional symmetry condition:

\textit{Find a lower bound on the boundary measure} (\textit{under uniform
enlargement}) \textit{of sets in $\Real^n$ that are transitive.}


The invariance under permutation condition rules out candidates like
the axis-parallel halfspaces
and one might expect that under this assumption, a~set should have
``large'' boundary. This intuition is confirmed by
Theorem~\ref{thmisolowerbound}. In this section we prove a
stronger version of this theorem that holds for all Boltzmann
measures.
\begin{Theorem}\label{thmisolowerboundgen}
Consider the product spaces $(\mathbb R^n, \mu_\rho^{\otimes n})$,
where $\mu_\rho$ denotes the Boltzmann measure with parameter
$\rho> 1$. There exists a constant $c =c(\rho)> 0$
such that the following holds for all $n \ge1$:

For any transitive Borel-measurable set $A \subset\Real^n$,
we have
\[
\liminf_{r \downarrow0} \frac{\mu_\rho^{\otimes n} (A+ [-r,r]^n)
- \mu_\rho^{\otimes n} (A)}{r} \geq ct(1-t) (\log n)^{ 1-
1/\rho},
\]
where $t = \mu_{\rho}^{\otimes n}(A)$.
\end{Theorem}

The
transitivity assumption on $A$ implies that
Theorem \ref{thmisolowerboundgen} is an immediate consequence
of Theorem \ref{corlogcocavehinf}, once we establish the
following lemma.
\begin{Lemma}\label{luniformenlargelog-concave}
Let $\lambda$ be a continuous symmetric log-concave density on
$\Real$. Let $A$ be any Borel-measurable subset of $\Real^n$. Then
\[
\liminf_{r \downarrow0} \frac{\nu^{\otimes n} (A+ [-r,r]^n) - \nu
^{\otimes n} (A)}{r}
\ge\sum_{i=1}^n I^h_i(A),
\]
where $h(x) = \lambda( \gL^{-1}(x) )$ for all $x \in[0,1]$.
\end{Lemma}
\begin{pf}
The proof is similar to the proof of
Proposition \ref{luniformenlarge}. For all $1 \leq i \leq n$,
define
\[
B_r^i = A +[-r,r]^{i-1} \times\{0\}^{n-i+1}.
\]
Like in the proof of Proposition \ref{luniformenlarge}, it is
sufficient to show that for each $i$,
%
%
\begin{equation}\label{Eq411}
\liminf_{ r \downarrow0} \frac{\nu^{\otimes n} (B_r^i +
\{0\}^{i-1}\times[-r, r]\times\{0\}^{n-i}) - \nu^{\otimes
n}(B_r^i)}{r} \ge I_i^h(A).\hspace*{-30pt}
\end{equation}
Note that for all $x \in\Real^n$, both $\nu^{\otimes n} (B_r^i)$
and $\nu((B_r^i)^x_i)$ are increasing as functions of $r$, and
thus they tend to some limit as $r \searrow0$. Furthermore, we
can assume that $\nu^{\otimes n} (\bar A \setminus A) = 0$, since
otherwise,
\[
\liminf_{r \downarrow0} \frac{\nu^{\otimes n} (A+ [-r,r]^n) -
\nu^{\otimes n} (A)}{r} \ge\liminf_{r \downarrow0}
\frac{\nu^{\otimes n} (\bar A \setminus A) }{r} \to\infty.
\]
Therefore,
\[
\nu^{\otimes n} (B_r^i) \searrow\nu^{\otimes n}(\bar A) =
\nu^{\otimes n}(A)
\]
and
%
%
\begin{equation}\label{Eq412}
\nu((B_r^i)^x_i) \searrow\nu(A^x_i)
\end{equation}
for almost every $x \in\Real^n$ (with respect to the measure
$\nu^{\otimes n}$).

Now observe that by the one-dimensional isoperimetric inequality
for symmetric log-concave distributions
(Theorem \ref{ThmIso-Bobkov}),
\begin{eqnarray*}
\nu^{\otimes n} (B_r^i + \{0\}^{i-1}\times[-r, r]\times
\{0\}^{n-i}) &=& \E_x \nu\bigl( (B_r^i)^x_i +[-r, r]\bigr)
\\
&\ge&\E_x \gL\bigl( \gL^{-1}(\nu( (B_r^i)^x_i))+ r\bigr).
\end{eqnarray*}
Therefore, using the mean value theorem like in the proof of
Proposition~\ref{luniformenlarge}, we get
%
%
\begin{eqnarray}\label{eqzeta}
&&\liminf_{ r \downarrow0} \frac{\nu^{\otimes n}
(B_r + \{0\}^{i-1}\times[-r, r]\times\{0\}^{n-i}) -
\nu^{\otimes n}(B_r)}{r} \nonumber\\[-8pt]\\[-8pt]
&&\qquad\ge\liminf_{ r \downarrow0}\E_x \inf_{ z \in[ \gL^{-1}(\nu
((B_r^i)^x_i)), \gL^{-1}(\nu( (B_r^i)^x_i))+ r ] } \lambda(z) .
\nonumber
\end{eqnarray}
Finally, by (\ref{Eq412}), for almost every $x \in\Real^n$,
\[
\lim_{r \downarrow0} \inf_{ z \in[ \gL^{-1}(\nu((B_r^i)^x_i)),
\gL^{-1}(\nu( (B_r^i)^x_i))+ r ] } \lambda(z) =
\lambda(\gL^{-1}(\nu(A^x_i))),
\]
and thus, by the dominated convergence theorem,
\[
\liminf_{ r \downarrow0}\E_x \inf_{ z \in[ \gL^{-1}(\nu
((B_r^i)^x_i)), \gL^{-1}(\nu( (B_r^i)^x_i))+ r ] } \lambda(z) = \E_x
\lambda( \gL^{-1}(\nu(A^x_i))) = I^h_i(A).
\]
This completes the proof of the lemma, and thus also the proof of
Theorem~\ref{thmisolowerboundgen}.
\end{pf}

\subsection{\texorpdfstring{Tightness of Theorems \protect\ref{corlogcocavehinf}, \protect\ref{ThmTalagrand} and \protect\ref{thmisolowerboundgen}}
{Tightness of Theorems 3.12, 3.13 and 4.1}}
\label{secsubtightness}

We conclude this section with showing that
Theorems \ref{corlogcocavehinf}, \ref{ThmTalagrand} and \ref
{thmisolowerboundgen} are tight (up to constant
factors) among sets with\vadjust{\goodbreak} constant measure, which we set for
convenience to be $1/2$. We demonstrate this by choosing an
appropriate sequence of ``\textit{one-sided boxes}.''\vspace*{-1pt}
\begin{Proposition}\label{lemtightnesslogconcave}
Consider the product spaces $(\mathbb R^n, \mu_\rho^{\otimes n})$,
where $\mu_\rho$ denotes the Boltzmann measure with parameter
$\rho\ge1$. Let $B_n := (-\infty, a_n ]^n$ where $a_n$ is chosen
such that $\Phi_\rho(a_n)^n = 1/2$. Then there exists a constant
$c=c(\rho)$ such that
\[
\ig_i(B_n) \le c \cdot\frac{(\log n)^{1 - 1/\rho}}{n}
\]
for all $1 \le i \le n$.\vspace*{-1pt}
\end{Proposition}
\begin{pf}
Fix an $i$. By elementary calculation,
\[
\ig_i(B_n) = \Phi_\rho(a_n)^{n-1} \phi_\rho(a_n) = (1/2)^{(n-1)/n}
\phi_\rho(a_n).
\]
Note that $1 - \Phi_\rho(a_n) \asymp n^{-1} $, and thus, by
Lemma \ref{lemboltzmanntail}, $a_n \asymp( \log n)^{1/\rho}$.
Furthermore, since by Lemma \ref{lemboltzmanntail},
$\phi_\rho(z) \asymp z^{\rho-1} (1- \Phi_\rho(z))$ for large $z$,
we have $ \ig_i(B_n) \asymp n^{-1}( \log n)^{1- 1/\rho}$, as
asserted.\vspace*{-1pt}
\end{pf}

The tightness of Theorem \ref{corlogcocavehinf} and
Theorem \ref{ThmTalagrand}(1) follows immediately
from Proposition \ref{lemtightnesslogconcave}. The tightness of
Theorem \ref{thmisolowerboundgen} follows using
Proposition~\ref{luniformenlarge} since $B$ is monotone.
The tightness of Theorem \ref{ThmTalagrand}(2) and the tightness in
$s$ in Theorem~\ref{ThmTalagrand}(3) and Theorem \ref{ThmTalagrand}(4) follows by
considering the subset $B_k \times\mathbb{R}^{n-k} \subset
\mathbb{R}^n$.\vspace*{-1pt}



\section{Geometric influences under rotation}
\label{secrotation}

Consider the product Gaussian measure $\mu^{\otimes n}$ on
$\mathbb R^n$. In Section \ref{secKKL} we obtained lower bounds
on the sum of geometric influences, and, in particular, we showed
that for a transitive set $A \subset\mathbb{R}^n$, the sum is at
least $\Omega(t(1-t) \sqrt{\log n})$, where $t=\mu^{\otimes
n}(A)$.

In this section we consider a different symmetry group, the group
of rotations of $\mathbb R^n$. The interest in this group comes
from the fact that the Gaussian measure is invariant under
rotations while the influence sum is not.

Indeed, a halfspace of measure $1/2$ may have influence sum as
small as of order~$1$ when it is aligned with one of the axis and
as large as of order $\sqrt{n}$ when it is aligned with the
diagonal direction $(1,1,\ldots,1)$.


In this section we show that under some mild conditions (that do
not contain any invariance assumption), rotation allows us to
increase the sum of geometric influences up to $\Omega(t(1-t)\sqrt
{-\log(t(1-t)) }
\sqrt{n})$. The dependence on $n$ in this lower bound is tight for
several examples, including
halfspaces and $L^2$-balls. We note that on the other extreme,
rotation cannot decrease the sum of geometric influences below
$\Omega(t(1-t)\sqrt{-\log(t(1-t))} )$, as follows from
a combination of Proposition \ref{luniformenlarge}, Lemma \ref
{lemincreasingequality}(ii) and the
isoperimetric inequality (\ref{EqBobkov}).\vspace*{-1pt}
\begin{definition}\label{defjcal}
Let $B(x, r): = \{ y \in\mathbb R^n\dvtx\| y - x\|_2 < r\}$ be the open
ball in
$\mathbb R^n$ with center at $x$ and radius $r$, and let $\bar{B}(x,
r)$ be the corresponding\vadjust{\goodbreak}
closed ball. For $\eps> 0$ and $A \subseteq R^n$, define
\[
A_\eps: = \{x \in A \dvtx\bar B(x, \eps) \cap A^c = \varnothing\}
\quad\mbox{and}\quad
A^\eps: = \{ x \in\mathbb R^n \dvtx B(x, \eps) \cap A
\ne\varnothing\}.
\]
Finally, denote by $\jcal_n$ the collection of all measurable sets $B
\subseteq\mathbb R^n$
for which there exists $\delta> 0$ such that for all $0< \eps< \delta
$, we have
%
%
\begin{equation}\label{eqpropertynice}
(B_\eps)^{2\eps} \supseteq B.
\end{equation}
An example of a subset of $\mathbb R^2$ that does not belong to the
class $\mathcal J_2$ is $ \{ (x,y)\dvtx1 \le x < \infty, 0 \le y < 1/x
\}$.
\end{definition}

The crucial ingredient in the proof Theorem \ref{ThmRotation} is a
lemma asserting
that under the conditions of the theorem, an enlargement of $A$ by a
\textit{random}
rotation of the cube $[-r,r]^n$ increases $\mu^{\otimes n}(A)$ significantly.
\begin{Notation}\label{Notofrac}
Let $O = O(n,\mathbb{R})$ be the set of all orthogonal
transformations on $\mathbb R^n$, and let $\pi$ be the (unique)
Haar measure on $O$. Denote by~$M$ a random element of $O$
distributed according to the measure $\pi$.
\end{Notation}
\begin{Lemma} \label{lemrandomrotation}
There exists a constant $K>0$ such that for
any $A \in\jcal_n$, we have
\[
\E_{M \sim\pi} \bigl[ \mu^{\otimes n}\bigl(A+M^{-1}(Kn^{-1/2} [-r,
r]^n) \bigr) \bigr] \ge\mu^{\otimes n}(A)+ \tfrac12 \mu^{\otimes n}(
A^{r/3}\setminus A)
\]
for all sufficiently small $r>0$ (depending on $A$).
\end{Lemma}

First we show that Lemma \ref{lemrandomrotation} implies
Theorem \ref{ThmRotation}.
\begin{pf}
Note that for any $g \in O$,
$g(A)$ is convex, and that $\mu^{\otimes n}$ is invariant under $g$.
Thus by Proposition \ref{luniformenlarge},\footnote{Note
that by Remark \ref{remconvex},
Proposition \ref{luniformenlarge} holds for convex sets.} we
have
\begin{eqnarray*}
\sum_{i=1}^n \ig_i(g(A)) &=&
\lim_{ r \downarrow0} \frac{ \mu^{\otimes n}(g(A)+[-r, r]^n) - \mu
^{\otimes n}(g(A))}{r}\\
&=& \lim_{ r \downarrow0} \frac{ \mu^{\otimes n}(A+g^{-1}([-r,
r]^n) ) - \mu^{\otimes n}(A)}{r}.
\end{eqnarray*}
Furthermore, note that for any $g \in O$,
\begin{eqnarray*}
\lim_{ r \downarrow0} \frac{ \mu^{\otimes n}(A+g^{-1}([-r, r]^n) )
- \mu^{\otimes n}(A)}{r} &\le&\lim_{ r \downarrow0} \frac{ \mu
^{\otimes n}(A+ \sqrt n [-r, r]^n ) - \mu^{\otimes n}(A)}{r} \\
&=&
\sqrt n \times\sum_{i=1}^n \ig_i(A).
\end{eqnarray*}
Therefore, by the dominated convergence theorem,
%
%
\begin{eqnarray}\label{Eq52}
&&\E_{M \sim\pi} \Biggl[ \sum_{i=1}^n \ig_i(M(A)) \Biggr] \nonumber\\[-8pt]\\[-8pt]
&&\qquad=
\lim_{ r \downarrow0} \frac{ \E_{M \sim\pi} [ \mu^{\otimes
n}(A+M^{-1}([-r, r]^n) ) ] -
\mu^{\otimes n}(A)}{r}.\nonumber
\end{eqnarray}
By Lemma \ref{lemrandomrotation}, we have (for a sufficiently
small $r$)
\[
\E_{M \sim\pi} \bigl[ \mu^{\otimes n}\bigl(A+M^{-1}(Kn^{-1/2}[-r,
r]^n) \bigr)\bigr] - \mu^{\otimes n}(A) \ge\tfrac12 \mu^{\otimes n}(
A^{r/3}\setminus A).
\]
By the standard Gaussian isoperimetric inequality,
\[
\mu^{\otimes n}( A^{r/3}\setminus A) \ge\mu\bigl( \bigl(-\infty,
\Phi^{-1}(t)+r/3 \bigr] \bigr).
\]
Substituting into (\ref{Eq52}), we get
\begin{eqnarray*}
\E_{M \sim\pi} \Biggl[ \sum_{i=1}^n \ig_i(M(A)) \Biggr] &\ge&
\limsup_{ r \downarrow0} \frac{\mu( (-\infty,
\Phi^{-1}(t)+K^{-1}n^{1/2}r/3 ] )}{2r} \\
&\geq& \frac{\sqrt
n}{6K} \phi(\Phi^{-1}(t) ) \\
&\ge& c t(1-t)\sqrt{-\log\bigl(t(1-t)\bigr)} \times
\sqrt n
\end{eqnarray*}
for some constant $c>0$ (where the last inequality follows from
the estimation given in Lemma \ref{lemboltzmanntail}, with
$\rho=2$). Thus there exists at least one orthogonal
transformation $g \in O$ such that
\[
\sum_{i=1}^n \ig_i(g(A)) \ge c t(1-t)\sqrt{-\log\bigl(t(1-t)\bigr)} \times
\sqrt n
\]
as asserted.
\end{pf}

Now we present the proof of Lemma \ref{lemrandomrotation}.
\begin{pf*}{Proof of Lemma \ref{lemrandomrotation}}
By Fubini's theorem, we have
%
%
\begin{eqnarray}\label{Eq53}\qquad
&&\E_{M \sim\pi} \bigl[ \mu^{\otimes n}\bigl(A+M^{-1}(Kn^{-1/2} [-r, r]^n) \bigr) \bigr]
\nonumber\\
&&\qquad= \E_{M \sim\pi} \bigl[ \mu^{\otimes n} \{ x \in\mathbb{R}^n \dvtx
x = y + z,
y \in A, z \in M^{-1}(K n^{-1/2} [-r, r]^n) \} \bigr] \\
&&\qquad= \E_{x \sim\mu^{\otimes n}} \bigl[ \pi\{ g \in O\dvtx x = y + z, y
\in A, z \in g^{-1}(K n^{-1/2} [-r, r]^n) \} \bigr].\nonumber
\end{eqnarray}
Since each $x \in A$ can be trivially represented as $y+z$ with $y =x
\in A, z =0
\in g^{-1}(K n^{-1/2} [-r, r]^n)$ for any $g \in O$,
the assertion of the lemma would follow immediately from
(\ref{Eq53}) once we show that for all $x \in A^{r/3}
\setminus A$,
%
%
\begin{equation}\label{Eq54}
\pi\{ g \in O: x = y + z, y \in A, z \in g^{-1}(K n^{-1/2} [-r,
r]^n) \} \geq1/2.
\end{equation}
Since $A \in\jcal_n$, we can choose $r$ sufficiently small such
that $A \subset(A_{r/3})^{2r/3}$, and thus $A^{r/3} \subset
(A_{r/3})^r$. Therefore, for any $x \in A^{r/3} \setminus A$,
there exists $y \in A_{r/3}$, such that $\|x-y\|_2<r$. If there
exists $y' \in B(y,r/3)$ such that $x-y' \in g^{-1}(K n^{-1/2}
[-r, r]^n)$, then $x$ can be represented as $y'+(x-y')$, as
required in the left-hand side of (\ref{Eq54}).
Therefore, it is sufficient to prove the following claim:
\begin{Claim}
For any $x,y \in\mathbb{R}^n$ such that $\|x-y\|_2<r$,
\[
\pi\{ g \in O\dvtx\exists y' \in B(y,r/3) \mbox{ such that }
x-y' \in g^{-1}(K n^{-1/2} [-r, r]^n) \} \geq1/2.
\]
\end{Claim}
\begin{pf}
Fix $x,y \in\mathbb{R}^n$
such that $\|x-y\|_2<r$. We have
\begin{eqnarray*}
&& \{ g \in O\dvtx\exists y' \in
B(y,r/3) \mbox{ such that } x-y' \in g^{-1}(K n^{-1/2} [-r,
r]^n) \} \\
&&\qquad = \{ g \in O\dvtx\exists y' \in
B(y,r/3) \mbox{ such that } g(x-y') \in K n^{-1/2} [-r, r]^n
\} \\
&&\qquad = \{ g \in O\dvtx\exists y'' \in
B(0,r/3) \mbox{ such that } g(x-y) - y'' \in K n^{-1/2} [-r,
r]^n \} \\
&&\qquad = \Bigl\{ g \in O\dvtx\inf_{y'' \in B(0,r/3)} \|g(x-y)
- y''\|_{\infty} \leq K n^{-1/2} r \Bigr\}.
\end{eqnarray*}
Note that
%
%
\begin{equation}\label{Eq55}
\pi\Bigl\{ g \in O\dvtx\inf_{y'' \in B(0,r/3)} \|g(x-y) -
y''\|_{\infty} \leq K n^{-1/2} r \Bigr\}
\end{equation}
is invariant under rotation of the vector $(x-y)$, and in
particular,
\[
\mbox{(\ref{Eq55})} = \pi\Bigl\{ g \in O\dvtx\inf_{y'' \in B(0,r/3)}
\bigl\|g(\|x-y\|_2 \times e_1) - y''\bigr\|_{\infty} \leq K n^{-1/2} r
\Bigr\},
\]
where $e_1 = (1, 0, \ldots, 0) \in\mathbb R^n$ is the unit vector
along the first coordinate axis.

A well-known property of the Haar measure says that if
$M \in O$ is distributed according to $\pi$, then any column of
$M$ is distributed like a normalized vector of independent
standard Gaussians. That is,
\[
M_{\mathrm{column}} \sim\frac{Z }{\|Z\|_2},
\]
where $Z = (Z_1,\ldots, Z_n)$ is a random $n$-vector with i.i.d.
standard Gaussian entries. Thus, $M(\|x-y\|_2 \times
e_1)$ is distributed like $\|x-y\|_2 \times Z/\|Z\|_2$. Therefore,
we have
\begin{eqnarray*}
\mbox{(\ref{Eq55})} &=& \prob_{Z \sim\mu^{\otimes n}} \biggl( \inf_{y''
\in B(0,r/3)} \biggl\| \|x-y\|_2 \times\frac{Z}{\|Z\|_2} -
y''\biggr\|_{\infty}
\leq K n^{-1/2} r \biggr) \\
& \geq&\prob_{Z \sim\mu^{\otimes n}} \biggl( \inf_{y''' \in
B(0,1/3)} \biggl\| \frac{Z}{\|Z\|_2} - y''' \biggr\|_{\infty} \leq K n^{-1/2}
\biggr).
\end{eqnarray*}
Note that if $Z \in\mathbb{R}^n$ satisfies
\[
\frac{\sum_i Z_i^2 1_{| Z_i | / \|Z\|_2 > Kn^{-1/2}}}{\|Z\|_2^2}
< 1/9,
\]
then the vector $y'''$ defined by $y'''_i = (Z_i \cdot1_{| Z_i |
/ \|Z\|_2 > Kn^{-1/2}})/\|Z\|_2$ satisfies
\[
y''' \in B(0,1/3) \quad\mbox{and}\quad \biggl\| \frac{Z}{\|Z\|_2}
- y'''\biggr\|_{\infty} \leq K n^{-1/2}.
\]
Hence,
\begin{eqnarray*}
\mbox{(\ref{Eq55})} &\geq&\prob_{Z \sim\mu^{\otimes n}} \biggl(
\inf_{y''' \in B(0,1/3)} \biggl\| \frac{Z}{\|Z\|_2} - y'''\biggr\|_{\infty}
\leq K n^{-1/2} \biggr) \\ &\geq&\prob_{Z \sim\mu^{\otimes n}}
\biggl(\frac{\sum_i Z_i^2 1_{| Z_i | / \|Z\|_2 >
Kn^{-1/2}}}{\|Z\|_2^2} < 1/9 \biggr).
\end{eqnarray*}
For $\gamma<1$ and $t>0$, by the Markov inequality,
\begin{eqnarray*}
\prob( \|Z \|^2 < \gamma n ) &\le&\prob\bigl( e^{-t \|Z \|^2} < e^{- t
\gamma n} \bigr) \le e^{ t \gamma n} (\E e^{-t Z_1^2} )^n \\
&=& e^{t \gamma n}
(1+2t)^{-n/2}.
\end{eqnarray*}
Optimizing over $t>0$, we have
%
%
\begin{equation} \label{Gaussiannormconcentration}
\prob( \|Z \|^2 < \gamma n ) \le(\gamma e^{1- \gamma})^{n/2}.
\end{equation}

Finally, again by the Markov inequality,
\[
\prob_{Z \sim\mu^{\otimes n}} \biggl[ \sum_{ i\dvtx|Z_i| > K/2} Z_i^2
\ge\frac{n}{450} \biggr] \le\frac{n \times[ \E Z_1^2 1_{\{ |Z_1|
>K/2 \} }] }{ n/450} \le1/4
\]
for sufficiently large $K>0$, and by (\ref{Gaussiannormconcentration}),
$\prob[\|Z\|^2_2 > n/50 ] \ge3/4$.
Therefore,
\begin{eqnarray*}
\mbox{(\ref{Eq55})} &\geq&\prob_{Z \sim\mu^{\otimes n}}
\biggl(\frac{\sum_i Z_i^2 1_{| Z_i |
/ \|Z\|_2 > Kn^{-1/2}}}{\|Z\|_2^2} < 1/9 \biggr) \\
& \geq&\prob_{Z \sim\mu^{\otimes n}} \biggl[ \biggl(\sum_{ i\dvtx|Z_i|
> K/2} Z_i^2 \le\frac{n}{450} \biggr) \wedge\bigl(\|Z\|_2 > \sqrt n/50\bigr)
\biggr] \\ & \geq&3/4+3/4-1=1/2.
\end{eqnarray*}
This completes the proof of the claim and of
Lemma \ref{lemrandomrotation}.
\end{pf}
\noqed
\end{pf*}

Intuitively, the condition $A \in\jcal_n$ means that the boundary of
$A$ is ``sufficiently smooth.'' One can easily check that if $A \in
\jcal_n$, then
the boundary of~A is a porous set and thus has Hausdorff dimension strictly
less than $n$ (see~\cite{zaj87} and references therein to know more about
porous sets). However, this condition is far from being sufficient.
Here we give a sufficient condition for a set to
belong to $\jcal_n$ in terms of smoothness of its boundary.
\begin{definition}
Let $A \subset\mathbb R^n$ be a measurable set. We write $\partial A
\in C^1$ and
say that the boundary of $A$ is of class $C^1$ if for any point
$z \in\partial A$, there exists $r =r(z)> 0$ and a one-to-one mapping
$\psi$ of $B(z, r)$ onto an open set $D = D \subseteq\mathbb R^n$
such that:
\begin{itemize}

\item$\psi\in C^1( \bar B(z, r))$ and $\psi^{-1} \in C^1( \bar D)$;

\item$\psi(B(z, r) \cap\partial A ) = D \cap\{x \in\mathbb R^n \dvtx
x_1 =0\}$;

\item$\psi(B(z, r) \cap\operatorname{int}(A)) \subseteq(0,\infty)
\times\mathbb{R}^{n-1}$.
\end{itemize}
\end{definition}
\begin{Proposition}
Let $A \subset\mathbb R^n$ be a bounded set with $\partial A \in
C^1$. Then $A \in\mathcal J_n$.
\end{Proposition}
\begin{pf}
Suppose on the contrary that $A \notin\mathcal J_n$. Then there
exists a sequence $\{x^m\}_{m=1}^{\infty} $ such that $x^m \in A$ but
$x^m \notin(A_{1/m})^{2/m}$. Since $A$ is bounded, the sequence
contains a
subsequence $\{x^{m_k}\}$ converging to a point $x^0$. Clearly,
$x^0 \in\partial A$.

Since $\partial A \in C^1$, we can define a new set of local
coordinates $(y_1,y_2, \ldots,y_n)$ [also denoted by $(y_1, y')$, where
$y' \in\mathbb{R}^{n-1}$], such that:
\begin{longlist}[(2)]
\item[(1)] the point $x^0$ is the origin with respect to the $y$-coordinates;

\item[(2)] there exists an open neighborhood
$(-\delta_0,\delta_0) \times U \subseteq\mathbb R \times\mathbb R^{n-1}$
containing the origin and a continuously differentiable function
$f\dvtx U \to\mathbb R_+$, such that in the $y$-coordinates,
\[
\partial A \cap[(-\delta_0, \delta_0) \times U] = \{(f(y'), y')\dvtx
y'\in U \}
\]
and
%
%
\begin{equation}\label{Eq5110}
\operatorname{int} A \cap[(-\delta_0, \delta_0) \times U] = \{(y_1, y')
\dvtx y' \in U, f(y') < y_1 < \gd_0 \}.
\end{equation}
\end{longlist}

By the construction of the new coordinates, $f(y') \geq0$ for all
$y' \in U$ and $f(0):= f(0,0,\ldots,0)=0$. Since $f \in C^1(U)$,
it follows that $\nabla f(0) = 0$. Hence, by the continuity of the
partial derivatives of $f$, there exists $r_0>0$ such that
$\|\nabla f(y')\|_\infty\le1/(3 \sqrt{n}) $ for all $y' \in
B_{n-1}(0,r_0) \subseteq U$.

Let $y^m = (y^m_1,(y^m)')$ be the representation of the point
$x^m$ in the $y$-coordinates. Find $m$ large enough such that
$1/m < \min\{\delta_0/10, r_0/10 \}$, and~$y^m$ lies within $A \cap
[0, \delta_0/2] \times B_{n-1}(0,r_0/2)$. Define
\[
z = (z_1, z_2,\ldots, z_n) = y^m + (1.5m^{-1},0, \ldots, 0,0).
\]
We claim that $B(z, 1/m) \subseteq A$. This would be a contradiction to the
hypothesis $y^m \notin(A_{1/m})^{2/m}$.

Note that by the choice of $m$, we have $z \in A$, and moreover,
%
%
\begin{eqnarray}\label{Eq5111}
\operatorname{dist} ( z, \partial A) &\geq&\operatorname{dist} \bigl(z,
\partial A \cap[(-\gd_0, \gd_0)
\times B_{n-1}(0, r_0)] \bigr) \nonumber\\[-8pt]\\[-8pt]
&=& \inf_{ y' \in B_{n-1}(0, r_0)} \bigl\|
\bigl(y^m_1+1.5m^{-1},(y^m)'\bigr) -(f(y'), y')\bigr\|_2.\nonumber
\end{eqnarray}
We would like to show that if $\|(y^m)'-y'\|_2$ is ``small,'' then
$|y^m_1+1.5m^{-1}-f(y')|$ is ``big,'' and thus in total, the right-hand
side of (\ref{Eq5111}) cannot be ``too small.''

Define $w_1 := y^m_1 + 1.5 m^{-1} - f((y^m)')$. Note that since
$y^m \in A$, it follows from (\ref{Eq5110}) that $w_1
\ge1.5 m^{-1}$. By the mean value theorem, for each $y' \in
B_{n-1}(0,r_0)$,
\begin{eqnarray*}
|f((y^m)') - f(y')| &\leq& \Bigl(\sup_{ y'' \in B_{n-1}(0,r_0)}
\|\nabla f(y'')\|_\infty\Bigr) \|(y^m)'-y'\|_1 \\[-2pt] &\le&
\frac{\|(y^m)'-y'\|_1}{3\sqrt{n}} \leq\frac{\|(y^m)'-y'\|_2}{3},
\end{eqnarray*}
and thus
\[
|y^m_1+1.5m^{-1}-f(y')|=\bigl|w_1-\bigl(f(y')-f((y^m)')\bigr)\bigr| \geq1.5m^{-1} -
\frac{\|(y^m)'-y'\|_2}{3}.
\]
Consequently, if $\|(y^m)'-y'\|_2 \geq4.5m^{-1}$, then
\[
\bigl\| \bigl(y^m_1+1.5m^{-1},(y^m)'\bigr) -(f(y'), y')\bigr\|_2 \geq
\|(y^m)'-y'\|_2 \geq4.5m^{-1},
\]
and if $\|(y^m)'-y'\|_2 < 4.5m^{-1}$, then
\begin{eqnarray*}
&&
\bigl\| \bigl(y^m_1+1.5m^{-1},(y^m)'\bigr) -(f(y'), y')\bigr\|_2 \\[-2pt]
&&\qquad\ge\sqrt{
\|(y^m)'-y'\|_2^2 + \biggl(1.5m^{-1} - \frac{\|(y^m)'-y'\|_2}{3}
\biggr)^2} \\[-2pt]
&&\qquad= \min_{0 \leq s < 4.5m^{-1}} \sqrt{ s^2 + (1.5m^{-1} -s/3)^2} =
\sqrt{\frac{81}{40}} m^{-1}.
\end{eqnarray*}
Combining the two cases, we get
\begin{eqnarray*}
\operatorname{dist}( z, \partial A) &\ge&\inf_{ y' \in B_{n-1}(0, r_0)}
\bigl\| \bigl(y^m_1+1.5m^{-1},(y^m)'\bigr) -(f(y'), y')\bigr\|_2 \\[-2pt] &\geq&\min
\Biggl( 4.5m^{-1},\sqrt{\frac{81}{40}} m^{-1} \Biggr) >1/m.
\end{eqnarray*}
This completes the proof.\vspace*{-3pt}
\end{pf}

If the condition $A \in\jcal_n$ is removed, we can prove only a
weaker lower bound on the maximal sum of geometric influences that
can be obtained by rotation.\vspace*{-3pt}
\begin{Proposition}\label{PropRotation}
Consider the product Gaussian measure $\mu^{\otimes n}$ on
$\mathbb R^n$. For any convex set $A$ with $\mu^{\otimes n}(A) =
t$, there exists an orthogonal transformation $g$ on
$\mathbb R^n$ such that
\[
\sum_{i=1}^n \ig_i(g(A)) \ge ct(1-t) \sqrt{-\log\bigl(t(1-t)\bigr)} \frac
{\sqrt n}{\sqrt{\log
n}},
\]
where $c> 0$ is a universal constant.\vadjust{\goodbreak}
\end{Proposition}

The proof of Proposition \ref{PropRotation} uses a weaker variant
of Lemma \ref{lemrandomrotation}:
\begin{Lemma}\label{lemmrandomrotation}
Let $M$ be as defined in Notation \ref{Notofrac}. There
exists a constant $K>0$ such that for any $A \subset\mathbb{R}^n$
and for any $r>0$, we have
\[
\E_{M \sim\pi} \bigl[ \mu^{\otimes n}\bigl(A+M^{-1}\bigl(K \sqrt{\log n}
\cdot n^{-1/2} [-r, r]^n\bigr) \bigr) \bigr] \ge\mu^{\otimes n}(A)+ \tfrac
12 \mu^{\otimes n}( A^{r}\setminus A).
\]
\end{Lemma}
\begin{pf}
By Fubini's theorem, we
have
\begin{eqnarray*}
&&\E_{M \sim\pi} \bigl[ \mu^{\otimes n}\bigl(A+M^{-1}\bigl(K \sqrt{ \log n}
\cdot n^{-1/2} [-r, r]^n\bigr) \bigr) \bigr] \\
&&\qquad= \E_{x \sim\mu^{\otimes n}} \bigl[ \pi\bigl\{ g \in O\dvtx x \in
A+g^{-1}\bigl(K \sqrt{\log n} \cdot n^{-1/2} [-r, r]^n\bigr) \bigr\} \bigr].
\end{eqnarray*}
Thus it is sufficient to prove that for any $x \in A^r \setminus
A$,
\[
\pi\bigl\{ g \in O\dvtx x \in A+g^{-1}\bigl(K \sqrt{\log n}
\cdot n^{-1/2} [-r, r]^n\bigr) \bigr\} \geq1/2.
\]
Equivalently, it is sufficient to prove that for any $x \in
B(0,r)$,
\[
\pi\bigl\{ g \in O\dvtx x \in g^{-1}\bigl(K \sqrt{\log n}
\cdot n^{-1/2} [-r, r]^n\bigr) \bigr\} \geq1/2.
\]
We can assume without loss of generality that $x=r' \cdot e_1$ for
some $r'<r$. By the argument used in the proof of
Lemma \ref{lemrandomrotation}, if $M \in O$ is distributed
according to $\pi$, then $M(r' \cdot e_1)$ is distributed like $r'
\cdot Z/\|Z\|_2$, where $Z = (Z_1,\ldots, Z_n)$ is a random
$n$-vector with i.i.d. standard Gaussian entries. Hence
%
%
\begin{eqnarray}\label{Eq56}
&&\pi\bigl\{ g \in O\dvtx x \in g^{-1}\bigl(K \sqrt{\log n}
\cdot n^{-1/2} [-r, r]^n\bigr) \bigr\} \nonumber\\
&&\qquad= \prob_{Z \sim\mu
^{\otimes n}} \biggl(
\biggl\|r' \frac{Z}{\|Z\|_2} \biggr\|_{\infty} \leq K \sqrt{\log n} \cdot
n^{-1/2} r \biggr) \nonumber\\[-8pt]\\[-8pt]
&&\qquad\geq\prob_{Z \sim\mu^{\otimes n}} \biggl( \biggl\| \frac{Z}{\|Z\|_2}
\biggr\|_{\infty} \leq K \sqrt{\log n} \cdot n^{-1/2} \biggr) \nonumber\\
&&\qquad\geq\prob_{Z \sim\mu^{\otimes n}} \bigl[ \bigl(\|Z\|_{\infty}
\leq K \sqrt{\log n}/\sqrt{50} \bigr) \wedge\bigl(\|Z\|_2 \ge\sqrt n/\sqrt{50}\bigr)
\bigr].\nonumber
\end{eqnarray}
We have
\begin{eqnarray*}
&&\prob_{Z \sim\mu^{\otimes n}} \bigl(\|Z\|_{\infty} \leq K \sqrt{\log
n}/\sqrt{50}\bigr) \\
&&\qquad\geq 1-n \prob\bigl(Z_i > K \sqrt{\log n}/\sqrt{50}\bigr)\\
&&\qquad\geq 1-
\frac{n}{\sqrt{2 \pi}} \cdot n^{-K^2/100} \geq3/4
\end{eqnarray*}
for a sufficiently big $K$. Therefore,
\[
\mbox{(\ref{Eq56})} \geq3/4+3/4-1=1/2,
\]
and this completes the proof of the lemma.
\end{pf}

The derivation of Proposition \ref{PropRotation} from
Lemma \ref{lemmrandomrotation} is the same as the derivation of
Theorem \ref{ThmRotation} from Lemma
\ref{lemrandomrotation}.\vadjust{\goodbreak}

Note that the convexity assumption on $A$ is used only to apply
Proposition \ref{luniformenlarge} that relates the sum of
influences to the size of the boundary with respect to uniform
enlargement. Thus, our argument also shows that for \textit{any}
measurable set $A$ with $\mu^{\otimes n}(A)=t$, there exists an
orthogonal transformation $g$ on $\mathbb R^n$ such that
\[
\lim_{ r \downarrow0} \frac{ \mu^{\otimes n}(g(A)+[-r, r]^n) -
\mu^{\otimes n}(g(A))}{r} \ge ct(1-t) \sqrt{-\log\bigl(t(1-t)\bigr)}
\frac{\sqrt n}{\sqrt{\log n}},
\]
where $c> 0$ is a universal constant.

Finally, we note that apparently the assertion of
Proposition \ref{PropRotation} is not optimal, and the
lower bound asserted in Theorem \ref{ThmRotation} should hold for
general convex sets.

\section*{Acknowledgments}

It is our pleasure to thank Omer Tamuz for several motivating
discussions that served as a starting point of the project. We would
also like to thank Franck Barthe, Lawrence Craig Evans, Steven N. Evans
and Gil Kalai for helpful suggestions. We also thank two anonymous
referees for careful reading of the manuscript and for making numerous
corrections and suggestions.


%

%
\printaddresses


\begin{thebibliography}{23}

\bibitem{bakry96}
%
\begin{barticle}[mr]
\bauthor{\bsnm{Bakry},~\bfnm{D.}\binits{D.}} \AND
\bauthor{\bsnm{Ledoux},~\bfnm{M.}\binits{M.}}
(\byear{1996}).
\btitle{L\'evy--{G}romov's isoperimetric inequality for an infinite-dimensional
diffusion generator}.
\bjournal{Invent. Math.}
\bvolume{123}
\bpages{259--281}.
\bid{doi={10.1007/s002220050026}, issn={0020-9910}, mr={1374200}}
\end{barticle}
%
\endbibitem

\bibitem{barthe04}
%
\begin{barticle}[mr]
\bauthor{\bsnm{Barthe},~\bfnm{F.}\binits{F.}}
(\byear{2004}).
\btitle{Infinite dimensional isoperimetric inequalities in product
spaces with
the supremum distance}.
\bjournal{J. Theoret. Probab.}
\bvolume{17}
\bpages{293--308}.
\bid{doi={10.1023/B:JOTP.0000020695.25095.c1}, issn={0894-9840}, mr={2053705}}
\end{barticle}
%
\endbibitem

\bibitem{bobkov96}
%
\begin{barticle}[mr]
\bauthor{\bsnm{Bobkov},~\bfnm{S.}\binits{S.}}
(\byear{1996}).
\btitle{Extremal properties of half-spaces for log-concave distributions}.
\bjournal{Ann. Probab.}
\bvolume{24}
\bpages{35--48}.
\bid{doi={10.1214/aop/1042644706}, issn={0091-1798}, mr={1387625}}
\end{barticle}
%
\endbibitem

\bibitem{bobkov97}
%
\begin{barticle}[mr]
\bauthor{\bsnm{Bobkov},~\bfnm{S.~G.}\binits{S.~G.}}
(\byear{1997}).
\btitle{Isoperimetric problem for uniform enlargement}.
\bjournal{Studia Math.}
\bvolume{123}
\bpages{81--95}.
\bid{issn={0039-3223}, mr={1438305}}
\end{barticle}
%
\endbibitem

\bibitem{Borell75}
%
\begin{barticle}[mr]
\bauthor{\bsnm{Borell},~\bfnm{Christer}\binits{C.}}
(\byear{1975}).
\btitle{The {B}runn--{M}inkowski inequality in {G}auss space}.
\bjournal{Invent. Math.}
\bvolume{30}
\bpages{207--216}.
\bid{issn={0020-9910}, mr={0399402}}
\end{barticle}
%
\endbibitem

\bibitem{bourgain92}
%
\begin{barticle}[mr]
\bauthor{\bsnm{Bourgain},~\bfnm{Jean}\binits{J.}},
\bauthor{\bsnm{Kahn},~\bfnm{Jeff}\binits{J.}},
\bauthor{\bsnm{Kalai},~\bfnm{Gil}\binits{G.}},
\bauthor{\bsnm{Katznelson},~\bfnm{Yitzhak}\binits{Y.}} \AND
\bauthor{\bsnm{Linial},~\bfnm{Nathan}\binits{N.}}
(\byear{1992}).
\btitle{The influence of variables in product spaces}.
\bjournal{Israel J. Math.}
\bvolume{77}
\bpages{55--64}.
\bid{doi={10.1007/BF02808010}, issn={0021-2172}, mr={1194785}}
\end{barticle}
%
\endbibitem

\bibitem{ehrhard83}
%
\begin{barticle}[mr]
\bauthor{\bsnm{Ehrhard},~\bfnm{Antoine}\binits{A.}}
(\byear{1983}).
\btitle{Sym\'etrisation dans l'espace de {G}auss}.
\bjournal{Math. Scand.}
\bvolume{53}
\bpages{281--301}.
\bid{issn={0025-5521}, mr={0745081}}
\end{barticle}
%
\endbibitem

\bibitem{Friedgut1}
%
\begin{barticle}[mr]
\bauthor{\bsnm{Friedgut},~\bfnm{Ehud}\binits{E.}}
(\byear{1998}).
\btitle{Boolean functions with low average sensitivity depend on few
coordinates}.
\bjournal{Combinatorica}
\bvolume{18}
\bpages{27--35}.
\bid{doi={10.1007/PL00009809}, issn={0209-9683}, mr={1645642}}
\end{barticle}
%
\endbibitem

\bibitem{Friedgut-Kalai}
%
\begin{barticle}[mr]
\bauthor{\bsnm{Friedgut},~\bfnm{Ehud}\binits{E.}} \AND
\bauthor{\bsnm{Kalai},~\bfnm{Gil}\binits{G.}}
(\byear{1996}).
\btitle{Every monotone graph property has a sharp threshold}.
\bjournal{Proc. Amer. Math. Soc.}
\bvolume{124}
\bpages{2993--3002}.
\bid{doi={10.1090/S0002-9939-96-03732-X}, issn={0002-9939}, mr={1371123}}
\end{barticle}
%
\endbibitem

\bibitem{Grimmett06}
%
\begin{barticle}[mr]
\bauthor{\bsnm{Graham},~\bfnm{B.~T.}\binits{B.~T.}} \AND
\bauthor{\bsnm{Grimmett},~\bfnm{G.~R.}\binits{G.~R.}}
(\byear{2006}).
\btitle{Influence and sharp-threshold theorems for monotonic measures}.
\bjournal{Ann. Probab.}
\bvolume{34}
\bpages{1726--1745}.
\bid{doi={10.1214/009117906000000278}, issn={0091-1798}, mr={2271479}}
\end{barticle}
%
\endbibitem

\bibitem{hatami09}
%
\begin{barticle}[mr]
\bauthor{\bsnm{Hatami},~\bfnm{Hamed}\binits{H.}}
(\byear{2009}).
\btitle{Decision trees and influences of variables over product probability
spaces}.
\bjournal{Combin. Probab. Comput.}
\bvolume{18}
\bpages{357--369}.
\bid{doi={10.1017/S0963548309009833}, issn={0963-5483}, mr={2501432}}
\end{barticle}
%
\endbibitem

\bibitem{kahn88}
%
\begin{bmisc}[auto:STB|2011-03-03|12:04:44]
\bauthor{\bsnm{Kahn},~\bfnm{J.}\binits{J.}},
\bauthor{\bsnm{Kalai},~\bfnm{G.}\binits{G.}} \AND
\bauthor{\bsnm{Linial},~\bfnm{N.}\binits{N.}}
\bhowpublished{The influence of variables on Boolean functions.
In \textit{Proceedings of 29th IEEE Symp. Foundations of Computer Science} (\textit{FOCS},
1988).}
\end{bmisc}
%
\endbibitem

\bibitem{kalai06}
%
\begin{bincollection}[mr]
\bauthor{\bsnm{Kalai},~\bfnm{Gil}\binits{G.}} \AND
\bauthor{\bsnm{Safra},~\bfnm{Shmuel}\binits{S.}}
(\byear{2006}).
\btitle{Threshold phenomena and influence: Perspectives from mathematics,
computer science, and economics}.
In \bbooktitle{Computational Complexity and Statistical Physics}
\bpages{25--60}.
\bpublisher{Oxford Univ. Press}, \baddress{New York}.
\bid{mr={2208732}}
\end{bincollection}
%
\endbibitem

\bibitem{keller09b}
%
\begin{barticle}[mr]
\bauthor{\bsnm{Keller},~\bfnm{Nathan}\binits{N.}}
(\byear{2011}).
\btitle{On the influences of variables on {B}oolean functions in product
spaces}.
\bjournal{Combin. Probab. Comput.}
\bvolume{20}
\bpages{83--102}.
\bid{doi={10.1017/S0963548310000234}, issn={0963-5483}, mr={2745679}}
\bptnote{check year}%
\end{barticle}
%
\endbibitem

\bibitem{margulis74}
%
\begin{barticle}[mr]
\bauthor{\bsnm{Margulis},~\bfnm{G.~A.}\binits{G.~A.}}
(\byear{1974}).
\btitle{Probabilistic characteristics of graphs with large connectivity}.
\bjournal{Problemy Peredachi Informatsii}
\bvolume{10}
\bpages{101--108}.
\bid{issn={0555-2923}, mr={0472604}}
\end{barticle}
%
\endbibitem

\bibitem{mossel09}
%
\begin{barticle}[mr]
\bauthor{\bsnm{Mossel},~\bfnm{Elchanan}\binits{E.}},
\bauthor{\bsnm{O'Donnell},~\bfnm{Ryan}\binits{R.}} \AND
\bauthor{\bsnm{Oleszkiewicz},~\bfnm{Krzysztof}\binits{K.}}
(\byear{2010}).
\btitle{Noise stability of functions with low influences: Invariance and
optimality}.
\bjournal{Ann. of Math. (2)}
\bvolume{171}
\bpages{295--341}.
\bid{doi={10.4007/annals.2010.171.295}, issn={0003-486X}, mr={2630040}}
\end{barticle}
%
\endbibitem

\bibitem{neyman33}
%
\begin{bmisc}[auto:STB|2011-03-03|12:04:44]
\bauthor{\bsnm{Neyman},~\bfnm{J.}\binits{J.}} \AND
\bauthor{\bsnm{Pearson},~\bfnm{E.~S.}\binits{E.~S.}}
(\byear{1933}).
\bhowpublished{On the problem of the most efficient tests of statistical
hypotheses. \textit{Philos. Trans. R. Soc. Lond. Ser. A Math. Phys. Eng.
Sci.} \textbf{231} 289--337}.
\end{bmisc}
%
\endbibitem

\bibitem{russo82}
%
\begin{barticle}[mr]
\bauthor{\bsnm{Russo},~\bfnm{Lucio}\binits{L.}}
(\byear{1982}).
\btitle{An approximate zero--one law}.
\bjournal{Z. Wahrsch. Verw. Gebiete}
\bvolume{61}
\bpages{129--139}.
\bid{doi={10.1007/BF00537230}, issn={0044-3719}, mr={0671248}}
\end{barticle}
%
\endbibitem

\bibitem{steele86}
%
\begin{barticle}[mr]
\bauthor{\bsnm{Steele},~\bfnm{J.~Michael}\binits{J.~M.}}
(\byear{1986}).
\btitle{An {E}fron--{S}tein inequality for nonsymmetric statistics}.
\bjournal{Ann. Statist.}
\bvolume{14}
\bpages{753--758}.
\bid{doi={10.1214/aos/1176349952}, issn={0090-5364}, mr={0840528}}
\end{barticle}
%
\endbibitem

\bibitem{sudakov78}
%
\begin{barticle}[mr]
\bauthor{\bsnm{Sudakov},~\bfnm{V.~N.}\binits{V.~N.}} \AND
\bauthor{\bsnm{Tsirelson},~\bfnm{B.}\binits{B.}}
(\byear{1974}).
\btitle{Extremal properties of half-spaces for spherically invariant measures}.
\bjournal{Zap. Nau\v cn. Sem. Leningrad. Otdel. Mat. Inst. Steklov. (LOMI)}
\bvolume{41}
\bpages{14--24, 165}.
\bid{mr={0365680}}
\end{barticle}
%
\endbibitem

\bibitem{Talagrand1}
%
\begin{barticle}[mr]
\bauthor{\bsnm{Talagrand},~\bfnm{Michel}\binits{M.}}
(\byear{1994}).
\btitle{On {R}usso's approximate zero--one law}.
\bjournal{Ann. Probab.}
\bvolume{22}
\bpages{1576--1587}.
\bid{issn={0091-1798}, mr={1303654}}
\end{barticle}
%
\endbibitem

\bibitem{zaj87}
%
\begin{barticle}[mr]
\bauthor{\bsnm{Zaj{\'{\i}}{\v{c}}ek},~\bfnm{L.}\binits{L.}}
(\byear{1987/88}).
\btitle{Porosity and {$\sigma$}-porosity}.
\bjournal{Real Anal. Exchange}
\bvolume{13}
\bpages{314--350}.
\bid{issn={0147-1937}, mr={0943561}}
\end{barticle}
%
\endbibitem

\end{thebibliography}
\end{document}